\documentclass[reqno]{amsart}
\usepackage{amssymb}
\usepackage{amsmath}
\usepackage[mathscr]{euscript}

\usepackage[small]{caption}
\usepackage{psfrag}
\usepackage{graphicx}
\usepackage{color}

\makeatletter
\@addtoreset{equation}{section}
\makeatother

\newtheorem{theorem}{Theorem}[section]
\newtheorem{lemma}[theorem]{Lemma}
\newtheorem{proposition}[theorem]{Proposition}
\newtheorem{corollary}[theorem]{Corollary}

\newcommand{\mf}[1]{{\mathfrak #1}}

\newcommand{\bb}[1]{{\mathbb #1}}

\newcommand{\ms}[1]{{\mathscr #1}}

\newcounter{as}[section]

\newtheorem{asser}[as]{Assertion}

\newcommand{\<}{\langle}
\renewcommand{\>}{\rangle}

\begin{document}

\title[Super-diffusive limit of exclusion processes] {A correction to
  the hydrodynamic limit of boundary driven exclusion processes in a
  super-diffusive time scale}

\author{E. Chavez, C. Landim}

\address{\noindent IMPA, Estrada Dona Castorina 110, CEP 22460 Rio de
  Janeiro, Brasil.  \newline e-mail: \rm \texttt{echavez@imca.edu.pe} }

\address{\noindent IMPA, Estrada Dona Castorina 110, CEP 22460 Rio de
  Janeiro, Brasil and CNRS UMR 6085, Universit\'e de Rouen, Avenue de
  l'Universit\'e, BP.12, Technop\^ole du Madril\-let, F76801
  Saint-\'Etienne-du-Rouvray, France.  \newline e-mail: \rm
  \texttt{landim@impa.br} }

\keywords{hydrodynamic equation, boundary driven interacting particle
  systems, first order correction, quasi-static transformations}

\begin{abstract}
  We consider a one-dimensional, weakly asymmetric, boundary driven
  exclusion process on the interval $[0,N]\cap \bb Z$ in the
  super-diffusive time scale $N^2 \epsilon^{-1}_N$, where $1\ll
  \epsilon^{-1}_N \ll N^{1/4}$. We assume that the external field and
  the chemical potentials, which fix the density at the boundaries,
  evolve smoothly in the macroscopic time scale. We derive an equation
  which describes the evolution of the density up to the order
  $\epsilon_N$.
\end{abstract}

\maketitle

\section{Introduction}
\label{sec0}

A theory of thermodynamic transformations for nonequilibrium
stationary states has been proposed recently \cite{bgjl-jsp12,
  bgjl-prl13} in the framework of the Macroscopic Fluctuation Theory
\cite{bdgjl-jsp02, bdgjl-rmp}. It defined the renormalized work performed
by a transformation between two nonequilibrium stationary states in
driven diffusive systems, and it proved a Clausius inequality which
postulates that the renormalized work is always larger than the
variation of the equilibrium free energy between the final and the
initial nonequilibrium states.

In quasi-static transformations, transformations in which the
variations of the environment are very slow, the renormalized work
coincides asymptotically with the variation of the equilibrium free
energy. More precisely, fix a transformation $u(t)$, $t\ge 0$, between
two nonequilibrium stationary states, and denote by $W^{\rm ren}(u)$
the renormalized work performed by $u$. Let $u_\epsilon$ be the
transformation $u$ slowed down by a parameter $\epsilon>0$,
$u_\epsilon (t)= u(t \epsilon)$. Then, $\lim_{\epsilon\to 0} W^{\rm
  ren}(u_\epsilon) = \Delta F$, where $\Delta F$ represents the
variation of the equilibrium free energy between the final and the
initial nonequilibrium states. Note that the asymptotic identity is
attained independently of the transformation $u$ chosen.

Let us mention that the theory of thermodynamic transformations
between nonequilibrium states, and the analysis of quasi-static
transformations has been extended to the framework of stochastic
perturbations of microscopic Hamiltonian dynamics in contact with heat
baths in \cite{o1, os1, ol1}.

To select, among the slow transformations between two nonequilibrium
stationary states, the one which minimizes the renormalized work we
have to examine the first order term in the expansion in $\epsilon$ of
the renormalized work. This question has been addressed in
\cite{bdgjl-jstat15}, where it was shown that for slow transformations
between two equilibrium states the first order correction of the
renormalized work is minimized by transformations whose intermediate
states are equilibrium states, and where a partial differential
equation which describes the evolution of the optimal transformation
has been derived.

A time-change permits to convert a slow transformation in an ordinary
transformation whose differential operator is multiplied by
$\epsilon^{-1}$. This observation brings us to the question of the
correction to the hydrodynamic equation of boundary driven interacting
particle systems. 

Consider a symmetric, one-dimensional dynamics in contact with
reservoirs and in the presence of an external field.  At the
macroscopic level the system is described by a local density $\rho(t,
x)$, $x\in [0,1]$, which evolves according to the driven diffusive
equation
\begin{equation}
\label{i01}
\left\{
\begin{split}
& \partial_t \rho = 
\partial_x(D(\rho)\partial_x\rho) - \partial_x
(\chi(\rho) E) \\
& f'(\rho(t,a))=\lambda_a(t) 
\quad\text{for}\quad a\;=\; 0\,,\, 1\;,
\end{split}
\right.
\end{equation}
where $D$ is the diffusivity, $\chi$ the mobility, $E(t,x)$ an
external field, $\lambda_0(t)$, $\lambda_1(t)$ time-dependent
chemical potentials, which fix the density at the boundaries, and $f$
the equilibrium free energy density.

For a fixed external field $E(x)$ and a chemical potential
$\lambda=(\lambda_0, \lambda_1)$, denote by $\bar\rho_{\lambda, E}$
the solution of the elliptic equation
\begin{equation}
\label{i04}
\left\{
\begin{split}
\partial_x(D(\rho)\partial_x\rho) - \partial_x (\chi(\rho) E)\;=\; 0
\;, \\ 
f'(\rho(a))=\lambda_a \quad\text{for}\quad a\;=\; 0\,,\, 1\;.
\end{split}
\right.
\end{equation}

Consider the driven diffusive equation \eqref{i01} speeded up by
$\epsilon^{-1}$. Fix a transformation $(\lambda(t), E(t))$,
$\epsilon>0$, and a bounded profile $v_0:[0,1]\to\bb R$. Denote by
$\rho_\epsilon (t)$ the solution of
\begin{equation}
\label{i06}
\left\{
\begin{split}
& \partial_t \rho = \epsilon^{-1} \big\{ 
\partial_x ( D(\rho) \partial_x \rho) - \partial_x
(\chi(\rho) E(t)) \big\} \;,  \\
& f'(\rho(t,0)) =  \lambda_0(t)\;, \;\; 
f'(\rho(t,1)) =  \lambda_1 (t)\;, \\
& \rho(0,\cdot) = \bar\rho_{\lambda(0), E(0)} (\cdot)
+ \epsilon v_0(\cdot)\;.
\end{split}
\right.
\end{equation}
A formal expansion in $\epsilon$ yields that, for $t>0$, $u_\epsilon
(t)=\epsilon^{-1} [\rho_\epsilon (t) - \bar\rho_{\lambda(t), E(t)}]$
converges to $v(t)$, the solution of the elliptic equation
\begin{equation}
\label{i05}
\left\{
\begin{split}
& \partial_t \bar\rho_{\lambda(t), E(t)} = 
\partial^2_x \big( D(\bar\rho_{\lambda(t), E(t)}) v \big) - \partial_x
\big(\chi'(\bar\rho_{\lambda(t), E(t)}) E(t) v \big) \;,  \\
& v(0) = v(1) = 0 \;.
\end{split}
\right.
\end{equation}
Note that the limit $v_t$ does not depend on the initial condition
$v_0$. 

The main results of this article, Theorems \ref{mt5} and \ref{mt7},
state a similar result for a microscopic dynamics speeded-up
super-diffusively. Consider a one-dimensional, weakly asymmetric,
exclusion process evolving on $\{1, \dots, N-1\}$, and in contact with
reservoirs at the boundaries. Assume that the density of each
reservoir evolves smoothly in the macroscopic time-scale, and that the
dynamics is speeded-up by $N^2 \epsilon^{-1}_N$, where $\epsilon_N\to
0$ as $N\uparrow\infty$. De Masi and Olla \cite{mo} proved that
starting from any initial distribution, at all macroscopic time $t>0$
the system converges to a local equilibrium state whose density
profile is given by the solution of the elliptic equation \eqref{i04}
with chemical potential $\lambda(t)$.

We examine in this article the correction to the hydrodynamic
equation. Assume that $\epsilon_N^{4} N\to\infty$, and that the
exclusion process starts from a local equilibrium state associated to
the density profile $\bar\rho_{\lambda(0), E(0)} + \epsilon_N
v_0$. Then, for all $t>0$, the system remains close, in the scale
$\epsilon_N^{-1}$, to a local equilibrium state whose density profile
is given by $\bar\rho_{\lambda(t), E(t)} + \epsilon_N v_t$, where
$v_t$ is the solution of the elliptic equation \eqref{i05}. More
precisely, for every cylinder function $\Psi$, and every continuous
function $H:[0,1]\to\bb R$, if $\eta^N_t$ represents the state at time
$t$ of the speeded-up exclusion process,
\begin{equation*}
\frac 1{\epsilon_N N} \sum_{j=1}^{N-1} H(j/N) \big\{ \tau_j \Psi(\eta^N_t) - 
E_{\bar\rho_{\lambda(t), E(t)} + \epsilon_N v_t}[\Psi] \big\} \to 0\;.
\end{equation*}
In this formula, $\{\tau_j: j\in \bb Z\}$ represents the group of
translations and $E_\gamma$ the expectation with respect to the local
equilibrium state associated to the density profile $\gamma$.

The proof of the main results follows the strategy proposed by
\cite{y, emy}, which consists in estimating the relative entropy of
the state of the process with respect to the local equilibrium state
whose density profile solves equation \eqref{i06} with
$\epsilon=\epsilon_N$. If $H_N(t)$ represents this latter relative
entropy, the main result asserts that for all $t>0$,
\begin{equation*}
\frac 1{N \epsilon_N^2} H_N(t) \;\to\;0\;.
\end{equation*}

The results presented in this article have a similarity to the
correction to the hydrodynamic equation, examined in \cite{emy, loy} in
the asymmetric case in dimension $d\ge 3$ and in \cite{jan} in the
symmetric case.

\section{Notation and main results}
\label{sec1}

\subsection{The model}
\label{bdep}

We examine a one-dimensional weakly asymmetric exclusion process in
contact with reservoirs. Fix $\Lambda=(0,1)$, and let $\Lambda_N =
\{1, \dots, N-1\}$, $N\ge 1$, be a discretization of $\Lambda$. The
microscopic point $j\in\Lambda_N$ thus represents the macroscopic
location $j/N\in \Lambda$.  Particles evolve on $\Lambda_N$ under an
exclusion rule which allows at most one particle per site.  The state
space is denoted by $\Sigma_N=\{0,1\}^{\Lambda_N}$, and the
configurations are represented by the Greek letters $\eta$, $\xi$ so
that $\eta(j)=1$ if site $j\in\Lambda_N$ is occupied for the
configuration $\eta$, and $\eta(j)=0$ otherwise.

Let $A_0$ be a finite subset of $\bb Z$ which contains the set
$\{0,1\}$. Consider a strictly positive function $c: \{0,1\}^{\bb Z}
\to \bb R_+$ which does not depend on the variables $\eta(0)$ and
$\eta(1)$ and whose support is contained in $A_0$:
\begin{equation*}
c(\eta) \;=\; c_\varnothing \;+\; \sum_{\substack{A\subset A_0 \\
    A\cap\{0,1\}=\varnothing}} c_A \prod_{k\in A} \eta(k)\;,
\end{equation*}
where $c_A$ are coefficients which may be negative. In the case where
$A_0 = \{0,1\}$, $c(\eta)$ is constant equal to
$c_\varnothing$. 

Denote by $\{\tau_k : k\in\bb Z\}$ the group of translations in
$\{0,1\}^{\bb Z}$ so that $\tau_k \eta$ is the configuration defined
by $(\tau_k \eta)(j) = \eta(k+j)$, $k$, $j\in\bb Z$. The action is
extended to cylinder functions $\Psi: \{0,1\}^{\bb Z} \to \bb R$, in
the usual way: $(\tau_k \Psi)(\eta) = \Psi(\tau_k \eta)$.

We assume throughout this article that the jump rate $c$ satisfies the
\emph{gradient} condition: There exist $m\ge 1$, cylinder functions
$h_1, \dots, h_m$, and finite-range, signed measures $\mu_1, \dots,
\mu_m$ on $\bb Z$ with vanishing total mass such that
\begin{equation}
\label{06}
[\eta(0) - \eta(1)] \, c(\eta) \;=\;
\sum_{a=1}^m \sum_{j\in \bb Z} \mu_a(j)\, (\tau_{-j} h_a)(\eta)\;.
\end{equation}
This decomposition is clearly not unique. In the case $c(\eta) = 1+
\eta(-1) + \eta(2)$, one may take $m=3$, $h_1(\eta)=\eta(-1) \eta(0)$,
$h_2(\eta) = \eta(0)\eta(2)$, $h_3(\eta)=\eta(0)$,
$\mu_1(0)=1=-\mu_1(2)$, $\mu_2(0)=1=-\mu_2(-1)$,
$\mu_3(0)=1=-\mu_3(-1)$.

Fix a chemical potential $\lambda:\partial\Lambda\to\bb R$, where
$\partial\Lambda$ represents the boundary of $\Lambda$. In one
dimension, $\lambda$ is simply a pair $(\lambda_0,\lambda_1)$.  Let
$\alpha=(\alpha_0$, $\alpha_1)$ be the density of particles associated
to the chemical potential $\lambda$: 
\begin{equation*}
\alpha_0 \;=\; \frac{e^{\lambda_0}}{1+e^{\lambda_0}} \;,\quad
\alpha_1 \;=\; \frac{e^{\lambda_1}}{1+e^{\lambda_1}}\;\cdot
\end{equation*}
Let $\tau^{N,\lambda}_j:\Sigma_N \to \{\alpha_0,\alpha_1, 0, 1\}^{\bb
  Z}$, $N\ge 1$, $j\in\bb Z$, be the operators defined by
\begin{equation*}
(\tau^{N,\lambda}_j\eta) (k) \;=\; \eta(k+j) \;\; \text{if $k+j\in
  \Lambda_N$}\;, \quad 
(\tau^{N,\lambda}_j\eta) (k) \;=\;
\begin{cases}
\alpha_0 & \text{if $k+j\le 0$}, \\
\alpha_1 & \text{if $k+j\ge N$}, \\
\end{cases}
\end{equation*}
for $k\in\bb Z$. As before the action of the operator
$\tau^{N,\lambda}_j$ can be extended to functions defined on
$\Sigma_N$.  For $N\ge 1$, $1\le j<N-1$, let the functions
$c^{N,\lambda}_{j,j+1} : \Sigma_N \to \bb R_+$ be given by
\begin{equation*}
c^{N,\lambda}_{j,j+1} \;=\; \tau^{N,\lambda}_j c  \;,
\end{equation*}
so that $c^{N,\lambda}_{j,j+1}(\eta) =
c(\tau^{N,\lambda}_j\eta)$. Note that $c^{N,\lambda}_{0,1}$ is usually
not equal to $c$. It follows from \eqref{06} that for $N\ge 1$, $1\le
j<N-1$,
\begin{equation}
\label{35}
w^{N,\lambda}_{j,j+1} \;:=\; [\eta(j) - \eta(j+1)] \, c^{N,\lambda}_{j,j+1}(\eta) \;=\;
\sum_{a=1}^m \sum_{k\in \bb Z} \mu_a(k)\, (\tau^{N,\lambda}_{j-k} h_a)(\eta)\;.
\end{equation}

We are now in a position to define the jump rates of the boundary
driven exclusion process.  Fix a smooth \emph{external field}
$E:[0,1]\to \bb R$, and let
\begin{equation*}
\begin{split}
& c^{N,\lambda,E}_{0,1} (\eta) \;=\;  r^\lambda_{0,1} (\eta) \, e^{(1/2N) \,
  E(0)\, [1 - 2\eta(1)]} \, c^{N,\lambda}_{0,1} (\eta)\;, \\
&\quad c^{N,\lambda,E}_{j,j+1} (\eta) \;=\; e^{(1/2N) \, E(j/N)\, [\eta(j) -
  \eta(j+1)]} \, c^{N,\lambda}_{j,j+1} (\eta)\;, \quad  1\le j\le N-2 \;, \\
&\qquad c^{N,\lambda,E}_{N-1,N} (\eta) \;=\; r^\lambda_{N-1,N} (\eta) \,
e^{(-1/2N) \,  E(1)\, [1 - 2\eta(N-1)]} \, c^{N,\lambda}_{N-1,N}(\eta) \;, 
\end{split}
\end{equation*}
where,
\begin{equation*}
\begin{split}
& r^\lambda_{0,1} (\eta) \;=\; \alpha_0 [1-\eta(1)] \;+\; \eta(1) [1-\alpha_0]\;, \\
&\quad r^\lambda_{N-1,N} (\eta) \;=\; \alpha_1 [1-\eta(N-1)] \;+\; \eta(N-1) [1-\alpha_1]\;.  
\end{split}
\end{equation*}

Denote by $L^{\lambda,E}_N=L_N$ the generator whose action on
functions $f:\Sigma_N \to \bb R$ is given by
\begin{equation}
\label{01}
(L_N f)(\eta)\;=\;  \sum_{j=0}^{N-1} c^{N,\lambda,E}_{j,j+1} (\eta) \, 
\{f(\sigma^{j,j+1} \eta) - f(\eta)\} \;.
\end{equation}
In this formula, the configuration $\sigma^{j,j+1} \eta$, $1\le j\le
N-2$, represents the configuration obtained from $\eta$ by exchanging
the occupation variables $\eta(j)$, $\eta(j+1)$,
\begin{equation*}
(\sigma^{j,j+1} \eta)(k) =
\begin{cases}
\eta(j+1),& k=j\;,\\
\eta(j), & k=j+1\;,\\
\eta(k), & k\neq j,j+1\;,\\
\end{cases}  
\end{equation*}
while $\sigma^{0,1} \eta$, $\sigma^{N-1,N} \eta$ represent the
configuration obtained from $\eta$ by flipping the occupation variables
$\eta(1)$, $\eta(N-1)$, respectively:
\begin{equation*}
(\sigma^j\eta)(k) =
\begin{cases}
\eta(k),& k\neq j\;, \\
1-\eta(k), & k=j\,.\\
\end{cases}  
\end{equation*}

\subsection{Transformations}

The dynamics introduced in the previous subsection is a finite-state,
irreducible, continuous-time Markov chain. It has therefore a unique
stationary state, denoted by $\nu^N_{\lambda, E}$. If the external
field $E(x)$ vanishes and the chemical potentials coincide, $\lambda_0
= \lambda_1=\lambda$, this stationary state is the Bernoulli product
measure with density $\rho = e^\lambda/(1+e^\lambda)$.

For a given continuous density profile $\gamma: [0,1] \to [0,1]$,
Denote by $\nu^N_{\gamma(\cdot)}$ the product measure on $\Sigma_N$
with marginals given by
\begin{equation}
\label{31}
\nu^N_{\gamma(\cdot)} \{\eta(j)=1\} \;=\; \gamma(j/N)\;, \quad
j\in\Lambda_N\;.
\end{equation}
Similarly, for $0\le\theta\le 1$, $\nu_{\theta}$, stands for the
Bernoulli product on $\{0,1\}^{\bb Z}$ with density $\theta$:
\begin{equation*}
\nu_{\theta} \{\eta(j)=1\} \;=\; \theta\;, \quad
j\in\bb Z\;.
\end{equation*}

To describe the macroscopic evolution of the density, denote the
\emph{diffusivity} by $D:[0,1]\to \bb R_+$, and the \emph{mobility} by
$\chi:[0,1]\to \bb R_+$:
\begin{equation}
\label{63}
D(\theta) \;=\; E_{\nu_\theta} [ c (\eta)]\; , \quad
\chi(\theta) \;=\; \frac 12 \, E_{\nu_\theta} \big[ \, [\eta(1)-\eta(0)]^2 c (\eta) \big]
\;=\; \theta(1-\theta)\, E_{\nu_\theta} [ c (\eta)] \;.
\end{equation} 
The transport coefficients $D$ and $\chi$ are related through the
local Einstein relation
\begin{equation}
\label{53}
D(\theta) \;=\; \chi(\theta) f''(\theta)\;,
\end{equation}
where $f:[0,1]\to\bb R$ the equilibrium free energy:
\begin{equation*}
f(\theta) \;=\; \theta \log \theta \;+\; [1-\theta] \log (1-\theta)\;.
\end{equation*}

Let $A=[0,1]$ or $\bb R_+$. Denote by $C^r(A)$, $r\ge 0$, the set of
functions $F:A\to\bb R$ which are $[r]$-times differentiable, where
$[r]$ stands for the integer part of $r$, and whose $[r]$-th
derivative is H\"older continuous with exponent $r-[r]$, and by
$C^r_0([0,1])$ the set of functions in $C^r([0,1])$ which vanish at
the boundary. If $r$ is an integer, we require the $[r]$-th derivative
to be continuous.  Similarly, $C^{r,s}(\bb R_+\times [0,1])$, $r$,
$s\ge 0$, represents the set of functions $F:\bb R_+\times [0,1]\to\bb
R$ which are $[r]$-times differentiable in the time variable,
$[s]$-times differentiable in the space variable and whose $[r]$-th
(resp. $[s]$-th) time (resp. space) derivative is H\"older continuous
with exponent $r-[r]$ (resp. $s-[s]$). As before, if $r$ or $s$ is an
integer, we require the corresponding derivative to be continuous.

Assume that $\lambda_a :\bb R_+\to \bb R$, $a=1$, $2$, are functions
in $C^1(\bb R_+)$ and that $E :\bb R_+\times [0,1]\to \bb R$ is a
function in $C^{1,2}(\bb R_+\times [0,1])$.  Fix a density profile
$\gamma: [0,1] \to (0,1)$ in $C^2([0,1])$ and assume that there exists
a function $\psi$ in $C^{1+\beta/2,2+\beta}(\bb R_+\times [0,1])$, for
some $\beta >0$, such that $f'(\psi(t,a))=\lambda_a(t)$ for $a=0$,
$1$, $\psi(0,x) = \gamma(x)$ for $x\in [0,1]$, and such that
\begin{equation*}
\partial_t \psi \;=\; \partial_x ( D(\psi) \partial_x \psi) 
- \partial_x (\chi(\psi) E(t))\quad
\text{at $(t,x)=(0,0)$ and $(t,x)=(0,1)$.}
\end{equation*}
Denote by $\rho(t,\cdot)$ the unique classical solution of the
parabolic equation
\begin{equation}
\label{27}
\left\{
\begin{split}
& \partial_t \rho = \partial_x ( D(\rho) \partial_x \rho) - \partial_x
(\chi(\rho) E(t))\;,  \\
& f'(\rho(t,0)) =  \lambda_0(t)\;, \;\; f'(\rho(t,1)) =  \lambda_1(t)\;, \\
& \rho(0,\cdot) = \gamma(\cdot)\;.
\end{split}
\right.
\end{equation}
We refer to Theorem 6.1 of Chapter V in \cite{lsu} for the existence
and the uniqueness of classical solutions of equation \eqref{27}.

Denote by $\ms M_N = \ms M (\Sigma_N)$ the set of probability measures
on $\Sigma_N$ endowed with the weak topology. For two probability
measures $\mu$, $\pi$ in $\ms M_N$, let $H_N(\mu|\pi)$ be the relative
entropy of $\mu$ with respect to $\pi$:
\begin{equation*}
H_N(\mu|\pi) \;=\; \sup_{f} \Big\{ \int f \, d\mu \;-\; \log \int e^f\,
d\pi\Big\}\;,
\end{equation*}
where the supremum is carried over all functions $f:\Sigma_N\to \bb
R$.  It is well known \cite{kl} that the relative entropy has an
explicit expression:
\begin{equation}
\label{33}
H_N(\mu|\pi) \;=\; 
\begin{cases}
\displaystyle \int \log \frac {d\mu}{d\pi} \; d\mu & \text{if
  $\mu \ll \pi$}, \\
\infty & \text{otherwise.}
\end{cases}
\end{equation}

Denote by $L_N(t)$, $t\ge 0$, the generator $L_N$ introduced in
\eqref{01} in which the pair $(E,\lambda)$ is replaced by
$(E(t),\lambda(t))$, and by $\{S^N_t :t\ge 0\}$ the semigroup
associated to the generators $N^2 L_N(t)$: $(d/dt) S^N_t = L_N(t)
S^N_t$.  Note that time has been speeded-up diffusively since the
generator has been multiplied by $N^2$.

\begin{theorem}
\label{mt3}
Let $\{\mu_N: N\ge 1\}$ be a sequence of probability measures,
$\mu_N\in \ms M_N$, such that
\begin{equation*}
\lim_{N\to\infty} \frac 1{N} \, H_N( \mu_N | \nu^N_{\gamma(\cdot)})
\;=\;0\;. 
\end{equation*}
Then, for every $t>0$,
\begin{equation*}
\lim_{N\to\infty} \frac 1{N} \, H_N(\mu_N S^N_t | \nu^N_{\rho(t,\cdot)})
\;=\;0\;.
\end{equation*}
\end{theorem}

\begin{corollary}
\label{mt4}
Under the assumptions of Theorem \ref{mt3}, for every $t\ge 0$, every
continuous function $H:[0,1]\to \bb R$, and every cylinder function
$\Psi:\{0,1\}^{\bb Z} \to \bb R$,
\begin{equation*}
\lim_{N\to\infty} E_{\mu_N S^N_t} \Big[ \, \Big| \frac 1N \sum_{k=1}^{N-1} H(k/N)
\, (\tau^{N,\lambda(t)}_k \Psi)(\eta) - \int_0^1 H(x)\,
E_{\nu_{\rho(t,x)}} [\Psi] \, dx \Big| \, \Big] \;=\; 0\;.
\end{equation*}
\end{corollary}

\subsection{Quasi-static transformations}

Fix $\nu>0$, a function $\lambda$ in $C^1(\bb R_+)$, and let $\alpha:
\bb R_+\to (0,1)$ be given by
\begin{equation}
\label{59}
\alpha(t) \;=\; f'(\lambda(t))\;,
\end{equation}
Fix a function $v_0 = v^\nu_0$ in $C^2_0([0,1])$, and assume that
there exists a function $\psi$ in $C^{1+\beta/2,2+\beta}(\bb R_+\times
[0,1])$, for some $\beta>0$, such that $\psi(t,0) = \psi(t,1) =
\alpha(t)$, $t\ge 0$, $\psi(0,x) = \alpha(0) + \nu^{-1} v_0(x)$, $x\in
[0,1]$, and
\begin{equation*}
\partial_t \psi \;=\; \nu\, \partial_x ( D(\psi) \partial_x \psi) 
\quad \text{at $(t,x)=(0,0)$ and $(t,x)=(0,1)$.}
\end{equation*}
This means that we assume that
\begin{equation*}
\alpha'(0) \;=\; \partial_x \Big\{ D \big (\alpha(0) + 
\nu^{-1} v_0(a) \big) \, \partial_x v_0(a) \Big\} \quad\text{for
  $a=0$, $1$}\;.
\end{equation*}
Denote by $\rho(t,x)=\rho_{\nu}(t,x)$ the unique classical solution of
the initial--boundary value problem
\begin{equation*}
\left\{
\begin{split}
& \partial_t \rho = \nu \, \partial_x ( D(\rho) 
\partial_x \rho) \;, \\
& \rho(t,0) = \rho(t,1) =  \alpha (t)\;, \\
& \rho(0,x) = \alpha(0) + \nu^{-1} v_0\;.
\end{split}
\right.
\end{equation*}

Let $u_\nu : \bb R_+ \times [0,1] \to \bb R$ be the function given by
\begin{equation*}
u_\nu (t,x) \;=\; \nu\, \big\{ \rho_\nu (t,x) - \alpha(t) \big\}  \;,
\end{equation*}
and, for each $t\ge 0$, let $v_t: [0,1] \to \bb R$ be the unique
solution of the linear elliptic equation
\begin{equation}
\label{60}
\left\{
\begin{split}
& \partial_x (D(\alpha(t)) \partial_x v_t) = \alpha'(t)\;, \\
& v_t(0)=v_t(1) =0\;.
\end{split}
\right.
\end{equation}

\begin{proposition}
\label{s16}
Assume that $\lambda$ belongs to $C^2(\bb R_+)$ and that $v_0$ belongs
to $C^4_0([0,1])$. Then, for each $t\ge 0$,
\begin{equation*}
\lim_{\nu\to\infty} \int_0^1 [u_\nu(t,x) - v_t(x)]^2\, dx\;=\; 0\;.
\end{equation*}
\end{proposition}

One can strengthen the topology in which the convergence occurs, but
we do not seek optimal conditions here.

Inspired by the previous result, consider a function $\lambda$ in
$C^1(\bb R_+)$, and let $\alpha: \bb R_+\to (0,1)$ be given by
\eqref{59}. Fix a sequence $\epsilon_N$ which vanishes as $N\to\infty$
and a function $\gamma=\gamma_N$ in $C^2_0([0,1])$. Assume that there
exists a function $\psi$ in $C^{1+\beta/2,2+\beta}(\bb R_+\times
[0,1])$, for some $\beta>0$, such that $\psi(t,0) = \psi(t,1) =
\alpha(t)$, $t\ge 0$, $\psi(0,x) = \alpha(0) + \epsilon_N \gamma(x)$,
$x\in [0,1]$, such that
\begin{equation*}
\alpha'(a) \;=\; \partial_x \Big\{ D \big (\alpha(a) + 
\epsilon_N \gamma (a) \big) \, \partial_x \gamma(a) \Big\} \quad\text{for
  $a=0$, $1$}\;.
\end{equation*}
Denote by $\rho_N(t,x)$ the solution of
\begin{equation}
\label{29}
\left\{
\begin{split}
& \partial_t \rho = \epsilon^{-1}_N \partial_x ( D(\rho) \partial_x
\rho) \;, \\
& \rho(t,0) = \rho(t,1) =  \alpha (t)\;, \\
& \rho(0,x) = \alpha(0) + \epsilon_N \gamma(x)\;.
\end{split}
\right.
\end{equation}

Denote by $\ms L_N(t)$ the generator $L_N$ introduced in \eqref{01}
with $E=0$ and $\lambda_0=\lambda_1=\lambda(t)$. Let $\{T^N_t :t\ge
0\}$ be the semigroup associated to the generator $\epsilon^{-1}_N N^2
\ms L_N (t)$. Note that time has been speed-up by $\epsilon^{-1}_N
N^2$.

\begin{theorem}
\label{mt5}
Assume that $\epsilon^4_N N\to \infty$, that $\lambda$ belongs to
$C^2(\bb R_+)$, and that $\gamma$ belongs to $C^4_0([0,1])$.  Let
$\{\mu_N: N\ge 1\}$ be a sequence of probability measures, $\mu_N\in
\ms M_N$, such that
\begin{equation}
\label{32}
\lim_{N\to\infty} \frac 1{N \epsilon_N^2} \, H_N( \mu_N | \nu^N_{\rho_N(0,\cdot)})
\;=\;0\;. 
\end{equation}
Then, for every $t>0$,
\begin{equation*}
\lim_{N\to\infty} \frac 1{N \epsilon_N^2} \, H_N(\mu_N T^N_t | 
\nu^N_{\rho_N (t,\cdot)}) \;=\;0\;.
\end{equation*}
\end{theorem}

\begin{corollary}
\label{mt6}
Under the assumptions of Theorem \ref{mt5}, for every $t\ge 0$, every
continuous function $H:[0,1]\to \bb R$, and every cylinder function
$\Psi:\{0,1\}^{\bb Z} \to \bb R$,
\begin{equation*}
\lim_{N\to\infty} E_{\mu_N T^N_t} \Big[ \, \frac 1{ \epsilon_N} \Big| 
\frac 1{N} \sum_{k=1}^{N-1} H(k/N)\,
(\tau^{N,\lambda(t)}_k \Psi)(\eta) - \int_0^1 H(x) 
\, E_{\nu_{v_N(t,x)}} [\Psi] \, dx
\Big| \, \Big] \;=\; 0\;,
\end{equation*}
where $v_N(t,x) = \alpha(t) + \epsilon_N v(t,x)$, $v(t,x)$ being the
unique classical solution of the elliptic equation \eqref{60}.
\end{corollary}

\section{Proof of the main results}

We present in Theorem \ref{mt7} below a general statement from which
one can easily deduce Theorems \ref{mt3} and \ref{mt5}.  For a fixed
chemical potential $\lambda=(\lambda_0,\lambda_1)$ and a continuous
external field $E:[0,1]\to\bb R$, denote by $\bar \rho_{\lambda,
  E}:[0,1]\to\bb R$ the solution of the elliptic equation
\begin{equation}
\label{57}
\left\{
\begin{split}
  & \partial_x ( D(\rho) \partial_x
  \rho) - \partial_x (\chi(\rho) E) \;=\;0 \;, \\
  & f'(\rho(0)) =  \lambda_0\;, \;\; f'(\rho(1)) =  \lambda_1 \;,
  \end{split}
\right.
\end{equation}

Fix sequences $\{\epsilon_N :N\ge 1\}$, $\{\ell_N :N\ge 1\}$ such that
$\ell_N\to\infty$, $\epsilon_N\to 0$. Consider a time-dependent
external field $E$ in $C^{1,2}(\bb R_+\times [0,1])$ and a
time-dependent chemical potential $\lambda(t)=(\lambda_0(t),
\lambda_1(t))$ such that $\lambda_0$, $\lambda_1 \in C^{1}(\bb
R_+)$. Fix a density profile $\gamma=\gamma_N$ in $C^2([0,1])$ and
assume that there exists a function $\psi$ in
$C^{1+\beta/2,2+\beta}(\bb R_+\times [0,1])$, $\beta>0$, such that
$f'(\psi(t,a))=\lambda_a(t)$ for $a=0$, $1$, $\psi(0,x) =
\bar\rho_{\lambda(0), E(0)} (x) + \epsilon_N \gamma(x)$ for $x\in
[0,1]$, and such that
\begin{equation*}
\partial_t \psi = \ell_N \Big\{ \, \partial_x ( D(\psi) \partial_x \psi) 
- \partial_x (\chi(\psi) E(t)) \Big\} \quad
\text{at $(t,x)=(0,0)$ and $(t,x)=(0,1)$.}
\end{equation*}
Denote by $\rho_N(t,\cdot)$ the unique weak solution of the parabolic
equation
\begin{equation}
\label{34}
\left\{
\begin{split}
  & \partial_t \rho = \ell_N \Big\{ \partial_x ( D(\rho) \partial_x
  \rho) - \partial_x (\chi(\rho) E) \Big\} \;, \\
  & f'(\rho(t,0)) =  \lambda_0(t)\;, \;\; f'(\rho(t,1)) =  \lambda_1(t) \;, \\
  & \rho(0,x) = \bar\rho_{\lambda(0), E(0)} + \epsilon_N \gamma(x)\;.
\end{split}
\right.
\end{equation}

In Theorem \ref{mt7} the following conditions on the solution of
equation \eqref{34} are needed: For every $T>0$, there exists
$0<\delta<1$ such that
\begin{equation}
\label{51}
\delta \;\le\; \rho_N(t,x) \;\le\; 1-\delta \quad\text{for all}\quad 
0\le x\le 1\;,\;\; 0\le t\le T\;,\;\; N\ge1\;.
\end{equation}
To explain the second condition, observe that we may rewrite the PDE
\eqref{34} as
\begin{equation*}
\partial_t \rho = \ell_N \partial_x \big\{  \chi(\rho) 
\big[ \partial_x f'(\rho) -  E  \big] \big\} 
\end{equation*}
because $\chi(\rho) f''(\rho) = D(\rho)$ by Einstein relation
\eqref{53}. Let
\begin{equation*}
F_N(t,x) \;=\; \partial_x f'(\rho_N(t,x)) \;-\;  E(t,x)
\end{equation*}
We assume that for every $T>0$, there exists a finite
constant $C_0$ such that for all $N\ge 1$, $0\le t\le T$, 
\begin{equation}
\label{52}
\Vert F_N (t) \Vert_\infty \;\le\; \frac{C_0}{\ell_N} \;, \quad
\Vert \partial_x F_N (t) \Vert_\infty \;\le\; 
\frac{C_0}{\ell_N} \;.
\end{equation}
Note that for this condition to be fulfilled at $t=0$, we need $\ell_N
\epsilon_N$ to be bounded:
\begin{equation}
\label{56}
\ell_N \, \epsilon_N \;\le\; C_0
\end{equation}
for some finite constant $C_0$.

Consider two non-decreasing sequences $K_N$, $J_N$. We write 
\begin{equation*}
\text{$K_N\ll J_N$ if $K_N/J_N\to 0$ as $N\to\infty$.}
\end{equation*}

Recall that we denote by $L_N(t)$ the generator $L_N$ introduced in
\eqref{01} with $E(t)$, $\lambda(t)$ in place of $E$, $\lambda$,
respectively.  Let $\{\mf S^N_t :t\ge 0\}$ be the semigroup associated
to the generators $\{\ell_N N^2 L_N (s):s\ge 0\}$: $(d/dt) \mf S^N_t =
\ell_N N^2 \mf S^N_t L_N (t)$.

\begin{theorem}
\label{mt7}
Consider a continuous external field $E(t,x)$ and a continuous
chemical potential $\lambda(t)=(\lambda_0(t), \lambda_1(t))$.  Assume
that $\gamma$ belongs to $C^2_0([0,1])$, that conditions \eqref{51},
\eqref{52}, \eqref{56} hold, and that $\epsilon^{-4}_N \ll N$. Let
$\{\mu_N: N\ge 1\}$ be a sequence of probability measures, $\mu_N\in
\ms M_N$, such that
\begin{equation}
\label{36}
\lim_{N\to\infty} \frac 1{N \epsilon_N^2} \, H_N( \mu_N | \nu^N_{\rho_N(0,\cdot)})
\;=\;0\;. 
\end{equation}
Then, for every $t>0$,
\begin{equation*}
\lim_{N\to\infty} \frac 1{N \epsilon_N^2} \, 
H_N(\mu_N \mf S^N_t | \nu^N_{\rho_N(t,\cdot)})
\;=\;0\;.
\end{equation*}
\end{theorem}

The proof of Theorem \ref{mt7} is divided in several steps. Fix a
density $\theta\in (0,1)$, and denote by $\nu_\theta=\nu^N_\theta$ the
product measure on $\Sigma_N$ with density $\theta$:
\begin{equation}
\label{03}
\nu_\theta (\eta) \;=\; \frac 1{Z_N(\theta)} \,\exp\Big\{f'(\theta)
\sum_{j=1}^{N-1} \eta_j \Big\} \;,
\end{equation}
where $Z_N(\theta)$ is the partition function which turns
$\nu_\theta$ into a probability measure, and
\begin{equation}
\label{08}
\beta \;:=\; f'(\theta)\; =\; \log \frac{\theta}{1-\theta}\;.
\end{equation}
We use the same notation $\nu_\theta$ to represent the the Bernoulli
product measure on $\{0,1\}^{\bb Z}$ with density $\theta$.

Let $L^2(\nu_\theta)$ be the space of functions $f:\Sigma_N \to\bb R$
endowed with the scalar product
\begin{equation*}
\<f\,,\,g\>_{\nu_\theta} \;=\; \int f(\eta)\, g(\eta) \, \nu_\theta (d\eta)\;.
\end{equation*}
Denote by $L^*_N=L^{\lambda,E,*}_N$ the adjoint in $L^2(\nu_\theta)$
of the generator $L_N$ introduced in \eqref{01}. A simple computation
shows that for all $f:\Sigma_N \to\bb R$,
\begin{equation}
\label{02}
(L^*_N f)(\eta) \;=\;   (L^*_{0,1}  f)(\eta) \;+\;
\sum_{j=1}^{N-2}  (L^*_{j,j+1}  f)(\eta)  \;+\;   (L^*_{N-1,N}  f)(\eta)\;,
\end{equation}
where, for $1\le j\le N-2$,
\begin{equation*}
\begin{split}
&  (L^*_{N-1,N} f)(\eta) \;=\; c^{N,\lambda,E}_{N-1,N} (\sigma^{N-1,N}\eta)\,
e^{\beta (1-2\eta_{N-1})} \, f(\sigma^{N-1,N}\eta) 
\;-\; c^{N,\lambda,E}_{N-1,N} (\eta)\, f(\eta) \;, \\
& \quad (L^*_{j,j+1}  f)(\eta)  \; =\; c^{N,\lambda,E}_{j,j+1} (\sigma^{j,j+1}\eta) \,
f(\sigma^{j,j+1}\eta) \; -\; c^{N,\lambda,E}_{j,j+1}(\eta) \, f(\eta)\;, \\
& \qquad (L^*_{0,1} f)(\eta) \;=\; c^{N,\lambda,E}_{0,1} (\sigma^{0,1} \eta)\,
e^{\beta (1-2\eta_1)} \, f(\sigma^{0,1}\eta)    
\;-\; c^{N,\lambda,E}_{0,1} (\eta) \, f(\eta) \;.
\end{split}
\end{equation*} 
In this formula, $\beta$ is the chemical potential associated to the
density $\theta$, which has been introduced in \eqref{08}.  It follows
from the previous formula that the adjoint of $L_N(t)$ in
$L^2(\nu_\theta)$, denoted by $L^*_N(t)$, is given by \eqref{02} with
$E$ and $\lambda$ replaced by $E(t)$ and $\lambda(t)$.  

\begin{proof}[Proof of Theorem \ref{mt7}]
Fix sequences $\ell_N$, $\epsilon_N$ satisfying the assumptions of the
theorem, and let $\gamma$ be a function in $C^2_0([0,1])$. Denote by
$\rho(t,x)=\rho_N(t,x)$ the solution of \eqref{34}. Consider a
sequence of probability measures $\{\mu_N: N\ge 1\}$, $\mu_N\in \ms
M_N$, satisfying \eqref{36}. Let $\alpha(t)=(\alpha_0(t)$,
$\alpha_1(t))$ be the density of particles associated to the chemical
potential $\lambda(t)$:
\begin{equation}
\label{37}
\alpha_0(t) \;=\; \frac{e^{\lambda_0(t)}}{1+e^{\lambda_0(t)}} \;,\quad
\alpha_1(t) \;=\; \frac{e^{\lambda_1(t)}}{1+e^{\lambda_1(t)}}\;\cdot
\end{equation}

Recall that $\{\mf S^N_t :t\ge 0\}$ represents the semigroup
associated to the generator $N^2 \ell_N L_N(t)$, and let
\begin{equation}
\label{38} 
f_t \;=\; \frac{d \mu_N \mf S^N_t}{d\nu_\theta}\;, \quad \psi_t \;=\;
\frac{d \nu^N_{\rho_N(t,\cdot)}}{d\nu_\theta}\;.
\end{equation}
A simple computation yields
\begin{equation}
\label{04}
\psi_t(\eta)\;=\; \frac {Z_N(\theta)}{Z_N(\rho(t))} \,
\exp\Big\{\sum_{j=1}^{N-1} \eta_j [f'(\rho(t,j/N)) - f'(\theta)] \Big\}
\;, 
\end{equation}
where $\rho(t,x)=\rho_N(t,x)$ is the solution of equation \eqref{34},
$Z_N(\theta)$, $f$ have been introduced in \eqref{03}, \eqref{53},
respectively, and $Z_N(\rho(t))$ is the normalizing constant given by
\begin{equation*}
Z_N(\rho(t)) \;=\; \exp\Big\{- \sum_{j=1}^{N-1} \log [1-\rho(t,j/N)] \Big\}
\;.
\end{equation*}
With this notation, in view of \eqref{33},
\begin{equation*}
H(\mu_N \mf S^N_t | \nu_{\rho_N(t,\cdot)}) \;=\; \int f_t \log \frac{f_t}{\psi_t}\,
d\nu_\theta\;.  
\end{equation*}
Moreover, an elementary computation shows that the density $f_t$
solves the Kolmogorov forward equation
\begin{equation}
\label{39}
\frac d{dt} f_t \;=\; N^2 \ell_N L^*_N (t) f_t\;.
\end{equation}
The proof of Theorem \ref{mt7} is divided in three steps.

\smallskip\noindent{\bf Step 1: Entropy production.}  A computation,
similar to the one presented in the proof of Lemma 1.4 in
\cite[Chapter 6]{kl}, yields that
\begin{equation}
\label{12}
\frac d{dt} H(\mu_N \mf S^N_t | \nu_{\rho_N(t,\cdot)}) \;\le\; \int 
\frac{ N^2 \ell_N L^*_N (t) \psi_t  - \partial_t \psi_t}{\psi_t} \,  
f_t  \, d\nu_\theta\;.
\end{equation}

Let $h$ and $g:\{0,1\}^{\bb Z}\to \bb R$ be the cylinder functions
given by
\begin{equation}
\label{07}
h(\xi) \;=\; \sum_{a=1}^m m_a\; h_a(\xi)\;, \quad
g(\xi) \;=\; \frac 12\, [\xi_1-\xi_0]^2\, c(\xi)\;,
\end{equation}
where $m_a = \sum_k k \mu_a(k)$. Recall the definition of the
operators $\tau^{N,\lambda}_j$, $N\ge 1$, $1\le j<N-1$, introduced
just above \eqref{35}, and let $\tau^{N}_j (t) =
\tau^{N,\lambda(t)}_j$.  A long, but straightforward, computation
which uses the identity \eqref{35}, yields that
\begin{equation*}
\begin{split}
\frac{ N^2 \ell_N L^*_N (t) \psi_t  - \partial_t \psi_t}{\psi_t}  
\;=\; \ell_N \{ I_1 \;+\; I_2 \;+\; I_3 \} \;+\; O_N(\ell_N) \;, 
\end{split}
\end{equation*}
where $O_N(\ell_N)$ represents an error absolutely bounded by
$C_0\ell_N$, $C_0$ being a finite constant independent of $N$, and
where 
\begin{equation*}
\begin{split}
I_1\; &=\; \sum_{j=1}^{N-2} G_1(t,j/N)\, 
(\tau^N_j (t) \, h)(\eta)
\; +\; \sum_{j=1}^{N-2} G_2(t,j/N)  \, (\tau^N_j (t) \, g)(\eta)  \\
&\qquad -\; \sum_{j=1}^{N-1} \frac{(\partial_t \rho)}{\chi(\rho)}
(t,j/N) \, [\eta(j) - \rho(t,j/N)] \;,
\end{split}
\end{equation*}
\begin{equation*}
\begin{split}
I_2\; &=\; N \, H_-(t) \,
\sum_{a=1}^m \sum_{k\in\bb Z} \mu_a(k) \sum_{j=1-k}^{0}  (\tau^N_j
(t)\, h_a)(\eta)  \\
& \qquad -\; N \, H_+(t)\, 
\sum_{a=1}^m \sum_{k\in\bb Z} \mu_a(k) \sum_{j=N-1-k}^{N-2} (\tau^N_j
(t)\, h_a)(\eta)  \;,
\end{split}
\end{equation*}
\begin{equation*}
I_3\; =\; N \, H_+(t) \, 
(\tau^N_{N-1} (t) \, c) (\eta)\, [\eta_{N-1} - \alpha_1(t)]
\;-\; N \, H_-(t) \, 
(\tau^N_{0} (t) \, c)(\eta)\, [\eta_1 - \alpha_0(t)] \;.
\end{equation*}
In these formulas, 
\begin{align*}
& G_1(t,x) \;=\; \partial_x \big\{ 
\partial_x f'(\rho(t,x)) - E(t,x) \big\} \;, \\
&\quad G_2(t,x) = \partial_x f'(\rho(t,x)) 
\big\{ \partial_x f'(\rho(t,x)) -
E(t,x) \big\}\;, \\
&\qquad H_-(t) = \partial_x f'(\rho(t,0)) -
E (t,0)\;, \quad H_+(t) = \partial_x f'(\rho(t,1)) - E (t,1)\;,
\end{align*}
and $\alpha_0(t)$, $\alpha_1(t)$ are the densities at the boundary,
defined in \eqref{37}. Note that $G_1 = \partial_x F_N$, $G_2 = F_N^2
+ E F_N$, $H_-(t)= F_N(t,0)$ and $H_+(t)= F_N(t,1)$. In particular, by
\eqref{52}, there exists a finite constant $C_0$ such that for all
$N\ge 1$, $0\le t\le T$,
\begin{equation}
\label{54}
\Vert G_1 (t) \Vert_\infty \;\le\; \frac{C_0}{\ell_N} \;, \quad
\Vert G_2 (t) \Vert_\infty \;\le\; \frac{C_0}{\ell_N} \;, \quad
\Vert H_\pm (t) \Vert_\infty \;\le\; 
\frac{C_0}{\ell_N} \;.
\end{equation}

For a cylinder function $\Psi: \{0,1\}^{\bb Z}\to\bb R$, let
$\hat\Psi:[0,1]\to\bb R$ be the polynomial given by
\begin{equation}
\label{40}
\hat \Psi(\theta) \;=\; E_{\nu_\theta}[\Psi]\;,
\end{equation}
where, we recall, $\nu_\theta$ is the Bernoulli product measure with
density $\theta$. By \eqref{07}, \eqref{63} and \eqref{62},
\begin{equation}
\label{55}
\hat h' (\theta) \;=\; D(\theta) \;,\quad
\hat g(\theta)\;=\; \chi(\theta)\;.
\end{equation}

We claim that, in the first two line of $I_1$, the replacement of
the cylinder functions $\tau^N_j (t)h$, $\tau^N_j (t) g$ by $\tau^N_j
(t) h - \hat h(\rho(t,j/N))$, $\tau^N_j (t) g - \hat g(\rho(t,j/N))$,
respectively, produces an error absolutely bounded by a finite
constant independent of $N$. Similarly, the replace-cement in the two
line of $I_2$ of the cylinder functions $\tau^N_j (t) h_a$, $j\sim 0$,
$\tau^N_k (t) h_a$, $k\sim N$, by $\tau^N_j (t) h_a - \hat
h_a(\alpha_0(t))$, $\tau^N_k (t) h_a - \hat h_a(\alpha_1(t))$ produces
an error of the same order.

Indeed, denote by $J_1$ (resp. $J_2$) the first line of $I_1$
(resp. the two lines of $I_2$) with the cylinder functions $\tau^N_j
(t) h$, $\tau^N_j (t) g$ (resp. $\tau^N_j (t) h_a$, $j\sim 0$,
$\tau^N_k (t) h_a$, $k\sim N$) replaced by $\hat h(\rho(t,j/N))$,
$\hat g(\rho(t,j/N))$ (resp. $\hat h_a(\alpha_0(t))$, $\hat
h_a(\alpha_1(t))$). In the expression of $J_2$, observe that $\sum_k k
\mu_a(k) = m_a$. For any Lipschitz-continuous function $G:[0,1]\to\bb
R$, and for any non-negative integers $p$, $q$,
\begin{equation*}
\sum_{j=p}^{N-q} G(j/N) \; =\; N \int_0^1 G(x)\, dx \;+\; O_N(1)\;,
\end{equation*}
where $O_N(1)$ represents an error absolutely bounded by a finite
constant independent of $N$.  It follows from this estimate, from an
integration by parts, and from the identities \eqref{53}, \eqref{55}
that $J_1 + J_2$ is absolutely bounded by a finite constant
independent of $N$, proving the claim.

An elementary computation gives that
\begin{equation*}
\begin{split}
\frac{(\partial_t \rho)}{\chi(\rho)} \;=\;
\big\{ \partial^2_x f'(\rho) - \partial_xE\big \}\, 
D(\rho)
\;+\; \big\{ [\partial_x f'(\rho)]^2 - E\,
\partial_x f'(\rho) \big\} \, \chi'(\rho) 
\;.
\end{split}
\end{equation*}
In conclusion, in view of \eqref{55}, up to this point, we have shown
that
\begin{equation}
\label{13}
\frac{ N^2 \ell_N L^*_N (t) \psi_t  - \partial_t \psi_t}{\psi_t}  
\;=\; \ell_N \{ \hat I_1 \;+\; \hat I_2 \;+\; I_3 \} \;+\; 
O(\ell_N) \;, 
\end{equation}
where
\begin{equation*}
\hat I_1(t,\eta)\;=\; \sum_{j=1}^{N-2} G_1(t,j/N) \, 
V_N(h ; t, j,\eta)  
\; +\; \sum_{j=1}^{N-2} G_2(t,j/N) \, V_N(g ; t, j,\eta)\;,
\end{equation*}
\begin{equation*}
\begin{split}
\hat I_2(t, \eta)\; & =\; N \,H_-(t) \,
\sum_{a=1}^m \sum_{k\in\bb Z} \mu_a(k) \sum_{j=1-k}^{0}  
\big\{ (\tau^N_j(t) h_a)(\eta)  - \hat h_a (\alpha_0(t)) \big\} \\
& -\; N \, H_+(t) \,
\sum_{a=1}^m \sum_{k\in\bb Z} \mu_a(k) \sum_{j=N-1-k}^{N-2}
\big\{ (\tau^N_j(t) h_a)(\eta) - \hat h_a (\alpha_1(t)) \big\}  \;,
\end{split}
\end{equation*}
and, for a cylinder function $\varphi:\{0,1\}^{\bb Z} \to \bb R$,
\begin{equation*}
V_N(\varphi; t, j,\eta) \;=\; (\tau^N_j (t) \varphi)(\eta) 
- \hat \varphi(\rho(t,j/N)) - 
\hat \varphi' (\rho(t,j/N)) [\eta_j - \rho(t,j/N)] \;.
\end{equation*}

\smallskip\noindent{\bf Step 2: A mesoscopic entropy estimate.}
Denote by $D(\bb R_+, \Sigma_N)$ the space of right-continuous
trajectories $x:\bb R_+\to\Sigma_N$ with left limits. For each
probability measure $\mu$ in $\ms M_N$, denote by $\bb P^N_{\mu}$ the
probability measure on $D(\bb R_+, \Sigma_N)$ induced by the Markov
chain with generator $\ell_N N^2 L_N(t)$ starting from the
distribution $\mu$. Expectation with respect to $\bb P^N_{\mu}$ is
represented by $\bb E^N_{\mu}$.

Recall that $\epsilon^{-4}_N \ll N$.  Let $M_N = \epsilon^{-2}_N$, and
fix a sequence $\{K_N : N\ge 1\}$ such that $M_N \ll K_N$, $M_N \, K_N
\ll N$.  Let
\begin{equation*}
\tilde I_{1,N}(t,\eta) \; =\; \sum_{j=K_N+1}^{N-K_N-1} G_1(t,j/N) 
\, V_N(\hat h;t,j,\eta) \;+\; 
\sum_{j=1}^{N-1} G_2(t,j/N) \, V_N(\hat g;t,j,\eta)\;,
\end{equation*}
where, for a smooth function $\hat \varphi:[0,1]\to\bb R$,
\begin{equation*}
V_N(\hat \varphi ;t,j,\eta)\;=\; \hat \varphi(\eta^{K_N}(j)) 
- \hat \varphi(\rho(t,j/N)) - 
\hat \varphi'(\rho(t,j/N)) [\eta^{K_N}(j) - \rho(t,j/N)] \;.
\end{equation*}
Note that in the definition of $\hat I_1(t,\eta)$ the sum is carried
over $1\le j\le N-1$, while in the definition of $\tilde I_1(t,\eta)$
it is carried over $K_N+1\le j\le N-K_N-1$. In view of \eqref{54},
this produces an error of order $K_N/\ell_N$ in the difference between
$\hat I_1(t,\eta)$ and $\tilde I_1(t,\eta)$.

By \eqref{54} and Lemma \ref{s02}, since $M_N K_N \ll N$,
\begin{equation*}
\lim_{N\to\infty} \frac{M_N\ell_N}N \, \bb E_{\mu_N} \Big[\, 
\int_0^t \{ \hat I_{1,N}(s,\eta_s) - 
\tilde I_{1,N}(s,\eta_s)\} \, ds \Big] \;=\; 0\;,
\end{equation*}
and 
\begin{equation*}
\lim_{N\to\infty} \frac{M_N \ell_N}N \, 
\bb E_{\mu_N} \Big[\, \int_0^t \hat I_2(s,\eta_s) \, ds 
\Big] \;=\; 0\;, \quad
\lim_{N\to\infty} \frac{M_N \ell_N}N \, 
\bb E_{\mu_N} \Big[\, \int_0^t I_3(s,\eta_s) \, ds 
\Big] \;=\; 0\;.
\end{equation*}
On the other hand, by definition of $M_N$, by \eqref{56} and by the
assumption on $\epsilon_N$, $M_N \ell_N = \epsilon^{-2}_N \ell_N \le
C_0 \epsilon^{-3}_N \ll N$.  Therefore, in view of \eqref{12},
\eqref{13},
\begin{align*}
\frac 1{N\varepsilon^{2}_N} \, H(\mu_N \mf S^N_t |
\nu_{\rho_N(t,\cdot)})  \; & \le\;
\frac 1{N\varepsilon^{2}_N} \, H(\mu_N | \nu_{\rho_N(0,\cdot)}) \\
\;& +\; \frac{M_N\ell_N}N \, \bb E_{\mu_N} \Big[\, 
\int_0^t \tilde I_{1,N}(s,\eta_s) \, ds \Big]
\;+\; R_N\;,
\end{align*}
where $R_N$ vanishes as $N\to\infty$.

\smallskip\noindent{\bf Step 3: A large deviations estimate.}  A
Taylor expansion up to the second order shows that $V_N(\hat
\varphi;t,j,\eta)$ is absolutely bounded by $C_0 [\eta^{K_N}(j) -
\rho(t,j/N)]^2$. The second term on the right hand side of the
previous equation is thus bounded above by
\begin{equation*}
C_0  \, \bb E_{\mu_N} \Big[\, \int_0^t 
\frac{M_N\ell_N}N \sum_{j=K_N+1}^{N-K_N-1} J(s,j/N)  
\, [\eta^{K_N}_s(j) - \rho(s,j/N)]^2 \, ds \Big]\;,
\end{equation*}
where $J(t,x) = |G_1(t,x)| + |G_2(t,x)|$. Since $M_N =
\epsilon^{-2}_N$, by the entropy inequality, the previous expression
is less than or equal to
\begin{equation*}
\begin{split}
& \frac{C_0}{AN\varepsilon^{2}_N} \int_0^t H(\mu_N \mf S^N_s 
| \nu_{\rho_N(s,\cdot)}) \, ds \\
& \;+\;  \int_0^t \frac{\varepsilon^{-2}_N}{AN} \log
E_{\nu_{\rho_N(s,\cdot)}} \Big[\, \exp\Big\{
A \ell_N \sum_{j=K_N+1}^{N-K_N-1} J(s,j/N)  
\, [\eta^{K_N} (j) - \rho(s,j/N)]^2 \Big\} \Big]\, ds
\end{split}
\end{equation*}
for every $A>0$. By H\"older's inequality and since
$\nu_{\rho_N(s,\cdot)}$ is a product measure, the second term o the
last sum is less than or equal to
\begin{equation*}
\int_0^t \frac{\varepsilon^{-2}_N}{ANK_N} \sum_{j=K_N+1}^{N-K_N-1} \log
E_{\nu_{\rho_N(s,\cdot)}} \Big[\, \exp\Big\{ A \ell_N  J(s,j/N) \,
K_N \, [\eta^{K_N} (j) - \rho(s,j/N)]^2 \Big\} \Big]\, ds 
\end{equation*}
By \eqref{51} and \eqref{54}, $\ell_N J(s,j/N) \le C_0$ and $\delta
\le \rho_N(s,x)\le 1-\delta$ for some $\delta>0$. Therefore, since
$\nu_{\rho_N(s,\cdot)}$ is the product measure with density
$\rho_N(s,\cdot)$, there exists $A_0$ such that for
\begin{equation*}
E_{\nu_{\rho_N(s,\cdot)}} \Big[\, \exp\Big\{ A C_0
K_N \, [\eta^{K_N} (j) - \rho(s,j/N)]^2 \Big\} \Big]\;\le\;  C'_0 
\end{equation*}
for all $0<A\le A_0$. The previous integral is therefore less than or
equal to $C_0 \varepsilon^{-2}_N/ AK_N \ll 1$. This proves
that there exists a finite constant $C_0$ such that
\begin{align*}
\frac 1{N\varepsilon^{2}_N} \, H(\mu_N \mf S^N_t |
\nu_{\rho_N(t,\cdot)})  \; &\le\;
\frac 1{N\varepsilon^{2}_N} \, H(\mu_N | \nu_{\rho_N(0,\cdot)}) \\
\;& +\;
C_0 \int_0^t \frac{1}{N\varepsilon^{2}_N} H(\mu_N \mf S^N_s |
\nu_{\rho_N(s,\cdot)}) \, ds \;+\; R_N\;,
\end{align*}
where $R_N$ vanishes as $N\to\infty$.  To conclude the proof of
Theorem \ref{mt7} it remains to apply Gronwall inequality.
\end{proof}

\begin{proof}[Proof of Theorem \ref{mt3}]
Set $\epsilon_N =\ell_N=1$. \eqref{51}. Conditions 
\eqref{52} and \eqref{56} are trivially satisfied. 
The assertion of Theorem \ref{mt3} follows therefore from Theorem
\ref{mt7}. 
\end{proof}

\begin{proof}[Proof of Theorem \ref{mt5}]
Assume that the external field vanishes: $E(t,x)=0$ and that the left
and right chemical potentials are equal, $\lambda_0(t) =
\lambda_1(t)$, $t\ge 0$. In This case the solution of the elliptic
equation \eqref{57} $\bar\rho_{\lambda, 0}$ is constant in space,
$\bar\rho_{\lambda, 0}(x) =\alpha$, where $\alpha = f'(\lambda)$.

Condition \eqref{51} for $N$ large enough follows from Proposition
\ref{s17}. Condition \eqref{52}, which can be read as conditions on
$\partial_x \rho_N$ and $\partial^2_x \rho_N$, follows from
Propositions \ref{s17} and \ref{s19}.
\end{proof}

\begin{proof}[Proofs of Corollary \ref{mt4} and \ref{mt6}]
The proofs are analogous to the one of Corollary 1.3, Chapter 6 in
\cite{kl}, provided we replace in the statement of Corollary \ref{mt6}
$\nu_{v_N(t,x)}$ by $\nu_{\rho_N(t,x)}$. However, since $\Psi$ is a
cylinder function,
\begin{equation*}
\frac 1{ \epsilon_N} \Big| \int_0^1 H(x) \big\{ 
\, E_{\nu_{v_N(t,x)}} [\Psi] - E_{\nu_{\rho_N(t,x)}} [\Psi]\big\} \, dx
\Big|\;\le\; \frac{C_0}{\epsilon_N} 
\int_0^1 \big| v_N(t,x) - \rho_N(t,x) \big| \, dx \;. 
\end{equation*}
By definition of $v_N$, the right hand side is equal to
\begin{equation*}
C_0 \int_0^1 \big| u_N(t,x) - v(t,x) \big| \, dx\;,
\end{equation*}
where $u_N(t)= \epsilon^{-1}_N [\rho_N(t) - \alpha(t)]$. It remains to
recall the statement of Proposition \ref{s16}.
\end{proof}

\section{Entropy estimates}

We adopt in this section the notation and the set-up introduced in the
previous one.  Recall from \eqref{03} that $\theta\in (0,1)$ is a
fixed parameter and that $\nu_\theta$ is the product measure on
$\Sigma_N$ with density $\theta$.  It is not difficult to show that
there exists a finite constant $C_0 = C_0(\theta)$ such that 
\begin{equation*}
\sup_\mu H(\mu | \nu_\theta) \;\le\; C_0 N\;,
\end{equation*}
where the supremum is carried over all probability measures $\mu$ on
$\Sigma_N$.

Fix a smooth function $\widehat\lambda : \bb R_+ \times [0,1]\to \bb
R$ such that $\widehat\lambda (t,a) = \lambda_a(t)$, $t\ge 0$, $a=0$,
$1$. Let $\alpha(t,x) =
e^{\widehat\lambda(t,x)}/[1+e^{\widehat\lambda(t,x)}]$, and denote by
$\nu^N_{\alpha(t)}$, $t\ge 0$, the product measure on $\Sigma_N$
associated to the density $\alpha(t,x)$:
\begin{equation*}
\nu^N_{\alpha(t)} (\eta) \;=\; \frac 1{\hat Z_N(t)} \,
\exp\Big\{\sum_{j=1}^{N-1} \eta_j \, f'(\alpha(t,j/N))  \Big\}
\;, 
\end{equation*}
where $\hat Z_N(t)$ is the normalizing constant given by
\begin{equation*}
\hat Z_N(t) \;=\; \exp\Big\{- \sum_{j=1}^{N-1} \log [1-\alpha(t,j/N)] \Big\}
\;.
\end{equation*}
Note that $\alpha(t,x)$ takes values in $(0,1)$. In particular, the
quantities introduced above are well defined.

Recall that $\{\mf S^N_t :t\ge 0\}$ represents the semigroup
associated to the generator $N^2 \ell_N L_N(t)$, $(d/dt) \mf S^N_t =
L_N(t) \mf S^N_t$, and that $H_N(\mu | \pi)$ stands for the relative
entropy of $\mu$ with respect to $\pi$. Denote by $D^{\alpha(t)}_N(\cdot)$,
$t\ge 0$, the functional which acts on functions $h: \Sigma_N \to \bb
R$, as
\begin{equation*}
\begin{split}
D^{\alpha(t)}_N(h) \; & =\; \sum_{j=1}^{N-2} \int
c^{N,\lambda(t)}_{j,j+1}  
(\eta) \, [h(\sigma^{j,j+1}\eta) 
- h(\eta)]^2\, d \nu^N_{\alpha(t)} (\eta) \\
 &+\; \int c^{N,\lambda(t)}_{0,1}(\eta) \, r^{\lambda(t)}_{0,1}(\eta)\,
[h(\sigma^{0,1}\eta) - h(\eta)]^2\, d \nu^N_{\alpha(t)} (\eta)\\
& +\;  \int c^{N,\lambda(t)}_{N-1,N}(\eta) \, r^{\lambda(t)}_{N-1,N}(\eta)\,
[h(\sigma^{N-1,N}\eta) - h(\eta)]^2\, d \nu^N_{\alpha(t)} (\eta)\;,
\end{split} 
\end{equation*}

\begin{lemma}
\label{s03}
Fix a sequence $\{\mu_N : N\ge 1\}$ of probability measures, $\mu_N
\in \ms M_N$.  For every $T>0$, there exists a finite constant $C_0$,
depending only on $E(t)$, $\alpha(t)$, $0\le t\le T$, such that for
all $0\le t\le T$,
\begin{equation*}
H_N(\mu_N \mf S^N_t | \nu^N_{\alpha(t)}) \;\le \; -\frac {N^2 \, \ell_N}2 \int_0^t 
D^{\alpha(s)}_N (\sqrt{g_s}) \, ds \;+\; C_0 \, N \,\ell_N \;,
\end{equation*}
where $g_t= g^N_t = d\mu_N \mf S^N_t/ d \nu^N_{\alpha(t)}$.
\end{lemma}

\begin{proof}
In this proof, $C_0$ represents a finite constant which may depend
only on $\theta$, $E(t)$, $\alpha(t)$, $0\le t\le T$, but not on $N$.

Fix a sequence $\{\mu_N : N\ge 1\}$ of probability measures, $\mu_N
\in \ms M_N$.  Recall the definition of $f_t=f^N_t$, introduced in
\eqref{38}, and let $\phi_t$, $t\ge 0$, be given by
\begin{equation*}
\phi_t \;=\; \frac{d \nu^N_{\alpha(t)}}{d\nu_\theta} \quad\text{so that}
\quad g_t \;=\; \frac {f_t}{\phi_t}\;.
\end{equation*}
By definition,
\begin{equation*}
H_N(t) \;:=\; H(\mu_N \mf S^N_t| \nu^N_{\alpha(t)})  \;=\;
\int f_t \log \frac{f_t}{\phi_t} \, d\nu_\theta\;, 
\end{equation*}
so that
\begin{equation*}
\frac d{dt} H_N(t) \:=\; 
N^2 \, \ell_N \int f_t \, L_N(t) \log \frac{f_t}{\phi_t} \, d\nu_\theta\;
\;-\; \int f_t \, \partial_t \log \phi_t \, d\nu_\theta\;.
\end{equation*}
The second term on the right hand side is clearly bounded by $C_0
N$. On the other hand, since $a \log b/a \le 2 \sqrt{a} (\sqrt{b} -
\sqrt{a})$, for all $a$, $b>0$, the first term on the right hand side
is less than or equal to
\begin{equation*}
2 N^2 \, \ell_N \int h_t \, L_N(t) \, h_t \, d \nu^N_{\alpha(t)} \;,
\end{equation*}
where $h_t =  \sqrt{g_t} = \sqrt{f_t/\phi_t}$.

Recall the definition of the generator $L_N(t)$ introduced in
\eqref{01}. Denote by $L^o_N(t)$ the piece of $L_N(t)$ which
corresponds to the sum for $j$ in the range $1\le j\le N-2$, and
denote by $L^b_N(t)$ the remaining two terms. A change of variables
$\eta' = \sigma^{j,j+1}\eta$, $1\le j\le N-2$, yields
\begin{equation*}
2 \int h_t \, L^o_N(t) \, h_t \, d \nu^N_{\alpha(t)} \;=\;
-\, \sum_{j=1}^{N-2} \int c^{N,\lambda(t)}_{j,j+1}(\eta) \, [h_t (\sigma^{j,j+1}\eta) 
- h_t (\eta)]^2\, d \nu^N_{\alpha(t)} (\eta)\; +\; R_N\;,
\end{equation*}
where $R_N$ is a remainder absolutely bounded by
\begin{equation*}
\frac {C_0}{N} \sum_{j=1}^{N-2} \int c^{N,\lambda(t)}_{j,j+1} (\eta) \, [h_t(\sigma^{j,j+1}\eta) 
+ h_t(\eta)]\, |h_t(\sigma^{j,j+1}\eta) - h_t(\eta)| \, d \nu^N_{\alpha(t)} (\eta)\;
\end{equation*}
for some finite constant $C_0$. By Young's inequality, and since
$g_t=h^2_t$ is a density with respect to $\nu^N_{\alpha(t)}$, the
previous expression is bounded by the sum of a term which can be
absorbed by the first term on the right hand side of the penultimate
displayed equation with a term bounded by $C_0/N$, that is,
\begin{equation*}
2 \int h_t \, L^o_N(t) \, h_t \, d \nu^N_{\alpha(t)} \;\le \;
-\, \frac 12\, \sum_{j=1}^{N-2} \int c^{N,\lambda(t)}_{j,j+1} (\eta) \, 
[h_t(\sigma^{j,j+1}\eta) - h_t(\eta)]^2\, d \nu^N_{\alpha(t)} (\eta)\; +\;
\frac{C_0}{N}\;\cdot 
\end{equation*} 

Since $\widetilde \lambda(t)$ is equal to $\lambda(t)$ at the boundary
of the interval $[0,1]$ 
\begin{equation*}
\frac{\nu^N_{\alpha(t)} (\sigma^{0,1}\eta)} {\nu^N_{\alpha(t)}
  (\eta)}\, \frac{r^{\lambda(t)}_{0,1}(\sigma^{0,1}\eta)}{
  r^{\lambda(t)}_{0,1} (\eta)} \; =\; 1 \;+\;  R_N\;,
\end{equation*}
where $R_N$ is absolutely bounded by $C_0/N$.  In view of this
identity, and with a similar computation to the one presented for the
interior piece of the generator, we conclude that
\begin{equation*}
\begin{split}
& 2 \int h_t \, L^b_N(t) \, h_t \, d \nu^N_{\alpha(t)} \; \le \;
-\, \frac 12\, \int c^{N,\lambda(t)}_{0,1}(\eta) \, r^{\lambda(t)}_{0,1}(\eta)\,
[h_t(\sigma^{0,1}\eta) - h_t(\eta)]^2\, d \nu^N_{\alpha(t)} (\eta)\\
&\qquad\quad -\; \frac 12\, \int 
c^{N,\lambda(t)}_{N-1,N}(\eta)\,  r^{\lambda(t)}_{N-1,N}(\eta)\,
[h_t(\sigma^{N-1,N}\eta) - h_t(\eta)]^2\, d \nu^N_{\alpha(t)} (\eta)\; +\;
\frac{C_0}{N^2}\;.
\end{split}
\end{equation*}

It follows from the previous estimates that
\begin{equation*}
H_N(t) \;-\; H_N(0) \;\le \; -\frac {N^2 \, \ell_N}2 \int_0^t 
D_N (s,\sqrt{g_s}) \, ds \;+\; C_0 \, N \,\ell_N \;, 
\end{equation*}
which concludes the proof of the lemma since $H_N(0) \le C_0N$, as
observed at the beginning of this section.
\end{proof}

For a positive integer $k$, denote by $\eta^k(j)$ the density of
particles in an interval of length $2k+1$ centered at $j$:
\begin{equation*}
\eta^k(j)\;=\; \frac 1{2k+1} \sum_{i\in I_k(j)\cap \Lambda_N} \eta(i)\;,
\end{equation*}
where $I_k(j) = \{j-k, \dots, j+k\}$. 

Recall the definition of the polynomial $\hat h: [0,1] \to \bb R$
given in \eqref{40}, where $h$ is a cylinder function, and recall the
definition of the probability measures $\bb P^N_\mu$ introduced at the
beginning of Step 2 in the previous section.

\begin{lemma}
\label{s02}
Let $G_N:\bb R_+ \times [0,1]\to\bb R$ (resp. $H_N:\bb R_+ \to\bb R$),
$N\ge 1$, be a sequence of functions in $C^{0,1}(\bb R_+ \times
[0,1])$ (resp. $C(\bb R_+)$) such that for all $T>0$,
\begin{gather*}
\sup_{N\ge 1} \sup_{0\le t\le T} \Vert G_N (t) \Vert_\infty \;<\;
\infty\;, \quad 
\sup_{N\ge 1} \sup_{0\le t\le T} \Vert H_N (t) \Vert_\infty \;<\;
\infty\;.
\end{gather*}
Let $h:\{0,1\}^{\bb Z} \to \bb R$ be a cylinder function. Fix a
sequence $\{\mu_N : N\ge 1\}$ of probability measures, $\mu_N \in \ms
M_N$.  Consider two sequences $M_N\uparrow\infty$ and
$K_N\uparrow\infty$ such that $M_N \ll K_N$, $M_N K_N \ll N$. Then,
for every $T>0$,
\begin{equation*}
\lim_{N\to\infty} M_N \, \bb E_{\mu_N} \Big[\, \int_0^T
\frac{1}N \sum_{j=K_N+1}^{N-1-K_N} G_N(s,j/N) \big\{ (\tau_j h)(\eta_s) - 
\hat h(\eta^{K_N}_s(j)) \big\} \, ds \Big] \;=\; 0\;,
\end{equation*}
and
\begin{equation*}
\begin{split}
& \lim_{N\to\infty} M_N \bb E_{\mu_N} \Big[ \int_0^T
H_N(s) \big\{ \tau^{N,\lambda(s)}_0 h(\eta_s) - 
\hat h(\alpha(s,0)) \big\} \, ds \Big] \;=\;
0\;, \\
& \quad \lim_{N\to\infty} M_N \bb E_{\mu_N} \Big[ \int_0^T
H_N(s) \big\{ \tau^{N,\lambda(s)}_N h(\eta_s) - 
\hat h(\alpha(s,1)) \big\} \, ds \Big] \;=\;
0\;.   
\end{split}
\end{equation*}
\end{lemma}

\begin{proof}
Fix $T>0$ and $0\le t\le T$.  Every cylinder function can be written
as a linear combination of the functions $\Psi_A = \prod_{j\in A}
\eta_j$, $A$ a finite subset of $\bb Z$. It is therefore enough to
prove the lemma for such functions. We present the details for
$h=\Psi_{\{0,1\}}$, it will be clear that the arguments apply to all
cases.

Fix a sequence of continuous function $G_N:\bb R_+ \times [0,1]\to\bb
R$ satisfying the assumptions of the lemma and note that $\hat
h(\theta) = \theta^2$ in the case where $h=\Psi_{\{0,1\}}$. It follows
from the assumptions of the lemma and from a summation by parts that
\begin{equation*}
\frac{1}N \sum_{j=K_N+1}^{N-1-K_N} G_N(s,j/N) \Big\{ \eta (j)\eta(j+1) - 
\frac 1{2K_N+1} \sum_{k=-K_N}^{K_N} \eta (j+k)\eta(j+k+1) \Big\}
\end{equation*}
in the time interval $[0,T]$ is absolutely bounded by a term of order
$K_N/N$. On the other hand, we may write the difference $(2K_N+1)^{-1}
\sum_{|k|\le K_N} \eta (j+k)\eta(j+k+1) - \hat h(\eta^{K_N}(j))$ as
\begin{equation*}
\frac 1{(2K_N+1)^2} \sum_{k,\ell} \eta (j+k) 
\, [\eta(j+k+1) - \eta(j+\ell)] \;+\; O\Big(\frac 1{K_N}\Big) \;,
\end{equation*}
where the sum is carried over all $k$, $\ell$ such that $|k|\le K_N$,
$|\ell|\le K_N$, $k\not = \ell$. The error term takes into account the
diagonal terms $k=\ell$. Denote by $V_{j,K_N}(\eta)$ the first term of
the previous formula.

In view of the former estimates, the first expectation appearing in
the statement of the lemma is equal to
\begin{equation*}
\bb E_{\mu_N} \Big[\, \int_0^T
\frac{1}N \sum_{j=K_N+1}^{N-1-K_N} G_N(s,j/N) \, V_{j,K_N}(\eta_s) \, ds
\Big]  \;+\; R_N \;,
\end{equation*}
where $R_N$ is a remainder absolutely bounded by $C_0 \{(K_N/N) +
(1/K_N)\}$. Here and below, $C_0$ is a finite constant which does not
depend on $N$, and which may change from line to line.
 
Recall the definition of the density $g_s$, introduced at the
beginning of the proof of Lemma \ref{s03}.  The first term of the
previous formula is equal to
\begin{equation}
\label{10}
\int_0^T ds\, \int \frac{1}N \sum_{j=K_N+1}^{N-1-K_N} G_N(s,j/N) \,
  V_{j,K_N}(\eta) \, g_s (\eta) \, \nu^N_{\alpha(s)}  (d\eta)\;. 
\end{equation}
Recall the definition of $V_{j,K_N}(\eta)$ and represent the previous
integral, denoted by $I$, as $(1/2) I + (1/2) I$. In one of the
halves, perform the change of variables $\eta' =
\sigma^{j+k+1,j+\ell}\eta$ to rewrite the previous expression as
\begin{equation*}
\begin{split}
\frac 12\, & \frac 1{(2K_N+1)^2} \sum_{k,\ell}  \int_0^T ds\, \int \frac{1}N 
\sum_{j=K_N+1}^{N-1-K_N} G_N(s,j/N) \, \eta (j+k) \; \times \\
& \times \, [\eta(j+k+1) - \eta(j+\ell)]
\{ g_s (\eta) - g_s (\sigma^{j+k+1,j+\ell}\eta) \}  \, \nu^N_{\alpha(s)}
(d\eta) \;+ \; R_N\;.
\end{split}
\end{equation*}
In this formula, $R_N$ is a remainder which appears from the change of
measures $\nu^N_{\alpha(s)} (\sigma^{j+k+1,j+\ell} \eta)/
\nu^N_{\alpha(s)} (\eta)$, and which is bounded by $C_0 K_N/N$.
Rewrite $g_s(\eta) - g_s(\eta')$ as $[\sqrt{g_s(\eta)} - \sqrt{g_s(\eta')}]\,
[\sqrt{g_s(\eta)} + \sqrt{g_s(\eta')}]$ and apply Young's inequality to
estimate the previous expression by
\begin{equation*}
\begin{split}
& \frac 1{4A} \, \frac 1{\widetilde K_N^2} \sum_{k,\ell} \int_0^T ds\, \int \frac{1}N 
\sum_{j} G_N(s,j/N)^2 \, \Big\{ \sqrt{g_s (\eta)} +
\sqrt{g_s (\sigma^{j+k,j+\ell}\eta)} \Big\}^2  \, \nu^N_{\alpha(s)}  (d\eta) \\
& +\; \frac A{4} \, \frac 1{\widetilde K_N^2} \sum_{k,\ell} \int_0^T ds\, \int \frac{1}N 
\sum_{j} \Big\{ \sqrt{g_s (\eta)} - 
\sqrt{g_s (\sigma^{j+k,j+\ell}\eta)} \Big\}^2  \, \nu^N_{\alpha(s)}  (d\eta) 
\end{split}
\end{equation*}
for every $A>0$. In this formula, $\widetilde K_N = 2K_N+1$.  Since
$g_s$ is a density with respect to $\nu^N_{\alpha(s)}$, the first term of the
previous expression is bounded by
\begin{equation*}
\frac {C_0}{A} \, \int_0^T ds\, \frac{1}N 
\sum_{j=K_N+1}^{N-1-K_N} G_N(s,j/N)^2 \;\le\; \frac{C_0}A\;\cdot
\end{equation*}
On the other hand, by the path lemma, explained in pages 94-95 of
\cite{kl} and in details below equation (3.7) in \cite{land92}, the
second term of the previous formula is bounded above by
\begin{equation*}
\frac{ C_0 A K^2_N} {N} \, \int_0^T ds\, D^{\alpha(s)}_N(\sqrt{g_s})\;.
\end{equation*}
Recall that in the path lemma, a change of variables $\eta' =
\sigma^{j,j+1} \sigma^{j+1,j+2} \cdots \sigma^{k-1,k}\eta$ is
performed. Usually, the Jacobian of this change of variables is equal
to $1$ because the reference measure is a homogeneous product
measure. In the present case, where the measure $\nu^N_{\alpha(s)} $
is a local equilibrium, the Jacobian is equal to $\exp\{h(\eta)\}$,
where $h$ is uniformly bounded by $K_N/N$.  By Lemma \ref{s03}, the
previous displayed equation is less than or equal to $C_0 A
(K_N/N)^2$. Optimizing over $A$, we conclude that \eqref{10} is
bounded by $C_0 K_N/N$.

To compete the proof of the first assertion of the lemma it remains to
recollect all the previous estimates and to recall the assumptions on
the sequences $M_N$ and $K_N$.

We turn to the second assertion. Here again, we present the proof for
the left boundary in the case where $h(\eta)=\eta_1\eta_2$. Note that,
by definition of $\tau^{N,\lambda}_0$, the case $h(\eta)=\eta_0\eta_1$
reduces to the case $h(\eta)=\eta_1$.

By definition of $g_s$, the expectation appearing in the statement of
the lemma is equal to
\begin{equation*}
\int_0^T ds \, H_N(s) \int \big\{\eta_1 \eta_2 - \alpha(s,0)^2 \big\} \,
g_s(\eta) \, \nu^N_{\alpha(s)} (d\eta) \;.
\end{equation*}
Fix $s$ and write the difference $E_{\nu^N_{\alpha(s)}} [\eta_1 \eta_2 g_s ] -
\alpha(s,0)^2$ as 
\begin{equation}
\label{11}
\begin{split}
& \Big\{ E_{\nu^N_{\alpha(s)}} \big[\eta_1 \eta_2 g_s \big] \;-\; 
E_{\nu^N_{\alpha(s)}} \big[ \eta_2 g_s \big] \alpha(s,0) \Big\} \\
&\qquad \quad \;+\;
\Big\{ E_{\nu^N_{\alpha(s)}} \big[ \eta_2 g_s \big] \alpha(s,0)
\;-\; E_{\nu^N_{\alpha(s)}} \big[ g_s \big] \alpha(s,0)^2 \Big\}\;.
\end{split}
\end{equation}
Since $1 = \eta_1 + (1-\eta_1)$, the first term inside braces can be
written as
\begin{equation*}
[1-\alpha(s,0)] \, E_{\nu^N_{\alpha(s)}} \big[\eta_1 \eta_2 g_s \big] \;-\; 
\alpha(s,0) \, E_{\nu^N_{\alpha(s)}} \big[ (1-\eta_1) \eta_2 g_s \big] \;.
\end{equation*}
Performing a change of variables $\eta'= \sigma^{0,1}\eta$ in the
first expectation, this difference becomes
\begin{equation*}
\alpha(s,1/N) \, \frac{(1-\alpha(s,0) )}{1-\alpha (s,1/N)} \, 
E_{\nu^N_{\alpha(s)}} \big[(1-\eta_1) \eta_2 g_s(\sigma^{0,1}\eta) \big] \;-\; 
\alpha(s,0) \, E_{\nu^N_{\alpha(s)}} \big[ (1-\eta_1) \eta_2 g_s \big] \;,
\end{equation*}
Since $|\alpha(s,1/N) - \alpha(s,0)|\le C_0/N$, the previous
expression is equal to
\begin{equation*}
\alpha(s,0) \, E_{\nu^N_{\alpha(s)}} \Big[(1-\eta_1) \eta_2 \big\{
g_s(\sigma^{0,1}\eta) \,-\,  g_s (\eta) \big\} \Big]\; +\; R_N\;,
\end{equation*}
where $R_N$ is a remainder bounded by $C_0/N$ in view of the
assumptions on the sequence $H_N$.  At this point, we may repeat the
arguments presented in the first part of the proof to bound the first
term by $C_0 \{D^{\alpha(s)}_N(\sqrt{g_s}) \}^{1/2}$, whose time
integral, in view of Lemma \ref{s03}, is bounded by $C_0 N^{-1/2}$.  A
similar argument permits to estimate the second term in
\eqref{11}. This completes the proof of the lemma.
\end{proof}

\section{The hydrodynamic equation}

We prove in this section Proposition \ref{s16} and some estimates,
stated below in Propositions \ref{s17} and \ref{s19}, on the solution
of equation \eqref{22}.  Recall the definition of the spaces
$C^k([0,1])$ and $C^k_0([0,1])$, $k\ge 1$, introduced just below
\eqref{53}.  Denote by $\Vert f\Vert_p$, $p\ge 1$, the $L^p$-norm of a
function $f:[0,1]\to\bb R$,
\begin{equation*}
\Vert f\Vert_p^p \;=\; \int_0^1 |f(x)|^p\, dx\;.
\end{equation*}

Fix $\nu>0$, a smooth function $\alpha : \bb R_+ \to (0,1)$, and an
initial condition $\rho_0$ in $C^4([0,1])$ such that
$\rho_0(0)=\rho_0(1) = \alpha(0)$. Denote by
$\rho(t,x)=\rho_{\nu}(t,x)$ the solution of the initial--boundary
value problem
\begin{equation}
\label{22}
\left\{
\begin{split}
& \partial_t \rho = \nu \, \partial_x ( D(\rho) 
\partial_x \rho) \;, \\
& \rho(t,0) = \rho(t,1) =  \alpha(t)\;, \\
& \rho(0,x) = \rho_0\;.
\end{split}
\right.
\end{equation}

\begin{proposition}
\label{s17}
For every $t_0\ge 0$, there exists $\nu_0<\infty$, such that for all
$\nu\ge \nu_0$, there exist positive constants $0<b<B<\infty$,
depending only on $D$, $\alpha (t)$, $0\le t\le t_0$, such that for
all $0\le t\le t_0$,
\begin{align*}
& \Vert \rho(t) - \alpha (t) \Vert^2_\infty \;\le\;
B \, e^{- b \nu \, t} \Vert \partial_ x\rho_0 \Vert^2_2 
\; +\; \frac{B}{\nu^2} \;, \\
&\quad \Vert (\partial_x \rho) (t) \Vert^2_\infty \;\le\;
B \, e^{-b \nu t} \Big\{ \Vert \partial^2_x\rho_0 \Vert^2_2  \;+\;
\Vert \partial_ x\rho_0 \Vert^4_4 \;+\; \frac 1\nu\,  \Vert 
\partial_ x\rho_0 \Vert^2_2 \Big\} \;+\; \frac{B}{\nu^2} \;.
\end{align*}
\end{proposition}

In this proposition, $e^{-b\nu t}$ corresponds to the speed of
convergence to equilibrium of the solution of \eqref{22} in the case
where the boundary condition $\alpha(t)$ does not change in time,
while $1/\nu^2$ stands for the relaxation time due to the evolution of
the boundary conditions.

% muito mal dito

\begin{proposition}
\label{s19}
Assume that $\rho_0 = \alpha(0) + \varepsilon v_0$, where $v_0$
belongs to $C^4_0([0,1])$.  For every $t_0\ge 0$, there exists
$\varepsilon_0 >0$ and $\nu_0<\infty$, depending on $D$, $v_0$,
$\alpha (t)$, $0\le t\le t_0$, such that for all
$\epsilon<\epsilon_0$, $\nu\ge \nu_0$, there exist positive constants
$B<\infty$, such that for all $0\le t\le t_0$,
\begin{equation*}
\Vert (\partial^2_x \rho) (t) \Vert^2_\infty \;\le\; B\, \big\{
\varepsilon^2 + \frac 1{\nu^4} \big\}\;.
\end{equation*}
\end{proposition}

The proof of these propositions is divided in a sequence of
assertions. The Poincar\'e's inequality plays a fundamental role in
the argument. It states that there exists a finite constant $K_1$ such
that for every $C^1([0,1])$ function $f$ which vanishes at some point
$x\in [0,1]$,
\begin{equation*}
\int_0^1 f(x)^2\, dx \;\le\; K_1 \int_0^1 [f'(x)]^2\, dx
\;. 
\end{equation*}

Throughout this subsection, $c_0$, $C_0$ represent small and large
constants which depend only on $K_1$ and $D$.

Let 
\begin{equation}
\label{20}
\beta_1(t) \;=\; \sup_{0\le s\le t} |\alpha'(s)|\;,\;\;
\beta_2(t) \;=\; \sup_{0\le s\le t} |\alpha^{\prime\prime}(s)| \;. 
\end{equation}

% \;\; \beta(t)^2 \;=\; \beta_1(t)^2 \;+\;\beta_2(t)^2 

\begin{asser}
\label{s04}
There exist positive constants $0<c_0<C_0<\infty$ such that for all
$t\ge 0$,
\begin{equation*}
\int_0^1 [\rho(t) - \alpha(t)]^2 \, dx \;\le 
e^{-c_0\nu t} \int_0^1 [\rho(0) - \alpha(0)]^2 \, dx
\;+\; \frac{C_0}{\nu^2} \, \beta_1(t)^2 \big( 1 - e^{-c_0\nu t} \big) \;.
\end{equation*}
\end{asser}

\begin{proof}
The proof follows classical arguments. Since $\rho(t)=\alpha(t)$ at the
boundary, an integration by parts and the fact that $\alpha(t)$ is
space independent yield that
\begin{equation*}
\frac 12\, \frac d{dt} \int_0^1 [\rho(t) - \alpha(t)]^2\, dx \;=\;
-\, \nu  \, \int_0^1 D(\rho(t))\, (\partial_x \rho(t))^2\, dx
\;-\;  \alpha'(t) \, \int_0^1 [\rho(t) - \alpha(t)]\, dx \;.
\end{equation*}
Since the diffusivity is bounded below by a strictly positive
constant, in the first term we may replace $D(\rho(t))$ by $c_0$ and
the identity by an inequality. By Poincar\'e's inequality, the
integral of $- (\partial_x \rho(t))^2$ is bounded by the integral of
$- K_1^{-1} [\rho(t) - \alpha(t)]^2$.  The second term on the right
hand side can be estimated by Young's inequality. One of the terms is
absorbed by what remained of the first term. The other one is
$(C_0/\nu) \alpha'(t)^2$.

Up to this point we have shown that
\begin{equation*}
\frac 12\, \frac d{dt} \int_0^1 [\rho(t) - \alpha(t)]^2\, dx \;\le\;
-\, c_0\, \nu  \, \int_0^1 [\rho(t) - \alpha(t)]^2\, dx
\;+\;  \frac{C_0}\nu \, \alpha'(t)^2  \;.
\end{equation*}
To complete the proof, it remains to apply Gronwall inequality.
\end{proof}

Let $d:[0,1]\to\bb R$ be a primitive of $D$, $d'=D$, and let
\begin{equation}
\label{26}
c_1 \;=\; \inf_{0\le \alpha\le 1} D(\alpha)\;, \quad
C_1 \;=\; \Vert (\log D)'\Vert_\infty\;.
\end{equation}

\begin{asser}
\label{s05}
Assume that $2K_1C_1\beta_1(t_0) < c_1\nu$ for some $t_0>0$. Then, there
exists a positive constants $C_0<\infty$ such that for all $0\le t\le
t_0$,
\begin{equation*}
\begin{split}
\int_0^1 [\partial_ xd(\rho(t))]^2\, dx \;\le\; e^{- a_\nu \, t} 
\int_0^1 [\partial_ xd(\rho(0))]^2\, dx
\; +\; \frac{C_0}{\nu} \int_0^t e^{-a_\nu (t-s)} \beta_1(s)^2 \, ds\;,
\end{split}
\end{equation*}
where $a_\nu = (c_1/K_1) \nu - 2 C_1 \beta_1(t_0)$.
\end{asser}

\begin{proof}
The proof is similar to the previous one. Adding and subtracting
$\alpha'(t)$ we have that
\begin{equation*}
\begin{split}
\frac 12\, \frac d{dt} \int_0^1  [\partial_ xd(\rho(t))]^2 \, dx
\;& =\; \int_0^1  \partial_ x d(\rho(t)) \,
\partial_ x \Big\{ D(\rho(t)) \big[ 
\nu \, \partial^2_x d(\rho(t)) - \alpha'(t) \big] \Big\} \, dx \\
& +\; \alpha'(t) \int_0^1 \partial_ x d(\rho(t)) \,
\partial_ x D(\rho(t)) \, dx \;.
\end{split}
\end{equation*}
Since $\alpha(t)=\rho(t,0)=\rho(t,1)$, $\alpha'(t) = \nu
\, \partial^2_x d(\rho(t,0)) = \nu \, \partial^2_x d(\rho(t,1))$. In
particular, we may integrate by parts the first term on the right hand
side. This operation yields a negative term and one involving
$\alpha'(t)$. This latter expression can be estimated through Young's
inequality. The first piece is absorbed into the negative term and the
second piece is bounded by $(C_0/\nu) \alpha'(t)^2$. Hence,
\begin{equation*}
\begin{split}
\frac 12\, \frac d{dt} \int_0^1  [\partial_ xd(\rho(t))]^2 \, dx
\;& \le \; - \frac {c_1 \nu}2 \, \int_0^1  
[\partial^2_ x d(\rho(t))]^2 \, dx \;+\; \frac {C_0}\nu \,
\alpha'(t)^2 \\
& +\; C_1 \, |\alpha'(t)| \,
\int_0^1 [\partial_ x d(\rho(t))]^2 \, dx \;.
\end{split}
\end{equation*}
Since $\int_0^1 \partial_ x d(\rho(t)) dx =0$, applying Poincar\'e
inequality to the first term on the right hand side, we obtain that
the last expression is bounded above by
\begin{equation*}
- \Big[ \frac{c_1 \, \nu}{2K_1} - C_1 \beta_1(t) \Big] \int_0^1  [\partial_ x d(\rho(t))]^2
\, dx  \;+\; \frac {C_0}\nu \, \beta_1(t)^2 \;.
\end{equation*}
To complete the proof, it remains to replace $\beta_1(t)$ by
$\beta_1(t_0)$ in the term inside brackets, getting an expression
which is positive by assumption, and to apply Gronwall inequality.
\end{proof}

\begin{lemma}
\label{s18}
Assume that $c_1\nu > 2K_1C_1\beta_1(t_0)$ for some $t_0>0$. Then,
there exist positive constants $0<c_0<C_0<\infty$ such that for all
$0\le t\le t_0$,
\begin{equation*}
\Vert \rho(t) - \alpha (t) \Vert^2_\infty \;\le\;
C_0 e^{- c_0 \nu \, t} e^{C_0 \beta(t_0) t} 
\int_0^1 [\partial_ xd(\rho(0))]^2\, dx
\; +\; \frac{1}{\nu^2} \frac{C_0\beta_1(t)^2}{1-(A_1\beta_1(t_0)/\nu)}\;,
\end{equation*}
where $A_1= 2K_1 C_1/c_1$. 
\end{lemma}

\begin{proof}
Assume that $2K_1C_1\beta_1(t_0) < c_1\nu$ for some $t_0>0$ and fix
$0<t\le t_0$. Since $\alpha(t)=\rho(t,0)$, by Schwarz inequality there
exists a finite constant $C_0$ such that for every $x\in [0,1]$,
\begin{equation*}
|\rho(t,x) - \alpha(t)|^2 \;\le\; C_0 \int_0^1 [\partial_x \rho (t)]^2
\, dx \;\le\; C_0 \int_0^1 [\partial_x d(\rho (t))]^2 \, dx \;.
\end{equation*}
To complete the proof, it remains to recall Assertion \ref{s05} and to
estimate the term $\beta_1(s)^2$ appearing in the time integral by
$\beta_1(t)^2$.
\end{proof}

Let $F_n$, $G_n:\bb R_+\to\bb R$, $n\ge 1$, be given by
\begin{equation}
\label{16}
F_n(t) \;=\; \int_0^1 [\partial_x d (\rho(t))]^{2n}\, dx\;, \quad
G_n(t) \;=\; \int_0^1 [\partial^2_x d (\rho(t))]^2  
\, [\partial_x d (\rho(t))]^{2n}\, dx\;.
\end{equation}

\begin{asser}
\label{s15}
For all $n\ge 2$, there exist positive constants $0<c_0<C_0<\infty$,
$b_0>0$, such that for all $0<b<b_0$, $t\ge 0$,
\begin{equation*}
\begin{split}
& F_n(t) \;+\; c_0 \, n^2 \, \nu \int_0^t G_{n-1}(s)\, e^{- b
  \nu (t-s)}\, ds \\
&\qquad\qquad \;\le\; e^{-b \nu t} \,  F_n(0) 
\;+\; C_0 \frac{n^2}\nu\, \beta_1(t)^2 \int_0^t F_{n-1}(s)\, e^{-b
  \nu (t-s)}\, ds \;.
\end{split}
\end{equation*}
\end{asser}

\begin{proof}
Since $\alpha'(t) = \partial_t \rho (t,1) = \nu \partial^2_x d (\rho
(t,1))$, adding and subtracting $\alpha'(t)$, and then integrating by
parts yield that
\begin{equation*}
\begin{split}
F'_n(t) \; &=\; - 2n (2n-1) \, \nu \int_0^1  D (\rho(t))\, 
[\partial_x d (\rho(t))]^{2n-2}\, [\partial^2_x d (\rho(t))]^2 \, dx \\
&+\;  2n (2n-1) \, \alpha'(t)  \int_0^1  D (\rho(t))\,
[\partial_x d (\rho(t))]^{2n-2}\, \partial^2_x d (\rho(t)) \, dx \\
& +\; 2n \, \alpha'(t)  \int_0^1  [\partial_x d (\rho(t))]^{2n-1}\, 
\partial_x D (\rho(t))\, dx\;.
\end{split}
\end{equation*}
Apply Young's inequality to the second term on the right hand side to
bound it by the sum of two terms. The first one can be absorbed by the
first term on the right hand side, and the second one is bounded by
$C_0 (n^2/\nu) \alpha'(t)^2 F_{n-1}(t)$. In the last term on the right
hand side, replace $\partial_x D (\rho(t))$ by $\partial_x [D
(\rho(t)) - D(\alpha(t))]$ and integrate by parts to obtain that it is
equal to
\begin{equation*}
- 2n(2n-1) \, \alpha'(t)  \int_0^1  [\partial_x d (\rho(t))]^{2(n-1)}\, 
\partial^2_x d (\rho(t))\, [D (\rho(t)) - D(\alpha(t))] \, dx\;.
\end{equation*}
Apply Young's inequality to bound this expression by the sum of two
terms. The first one can be absorbed by the first term on the
penultimate formula, while the second one is less than or equal to
$C_0 n^2 \alpha'(t)^2 \nu^{-1} F_{n-1}(t)$. Therefore,
\begin{equation*}
F'_n(t) \; \le \; - c_0\, n^2 \, \nu \, G_{n-1}(t)
\; +\; C_0 \, n^2 \, \frac{\alpha'(t)^2}{\nu}  F_{n-1}(t) \;.
\end{equation*}

Let $f(x) = [\partial_x d (\rho(t,x))]^n$. Since $\int_0^1 \partial_x
d (\rho(t,x)) \, dx =0$, there exists $x_0\in [0,1]$ such that
$\partial_x d (\rho(t,x_0))=0$, so that $f(x_0)=0$. We may therefore
apply Poincar\'e's inequality to $[\partial_x d (\rho(t,x))]^n$ to
obtain that
\begin{equation*}
F_n(t) \;=\; \int f(x)^2 \, dx \;\le\; K_1 \int f'(x)^2 \, dx \;=\;
K_1 n^2 G_{n-1}(t)\;.
\end{equation*}
It follows from the previous estimates that
\begin{equation*}
F'_n(t) \; \le \; - b_0\, \nu \, F_{n}(t) \;-\; c_0\, n^2 \, \nu \, G_{n-1}(t)
\; +\; C_0 \, n^2 \, \frac{\alpha'(t)^2}{\nu}  F_{n-1}(t) 
\end{equation*}
for some $b_0>0$. The same inequality remains in force for any
$0<b<b_0$. It remains to apply Gronwall inequality to complete the
proof.
\end{proof}

Iterating the inequality appearing in the previous assertion without
the term $G_{n-1}$ yields

\begin{asser}
\label{s06}
For all $n\ge 2$, there exist positive constants $0<c_0<C_0<\infty$,
$b_0>0$, such that for all $0<b<b_0$, $t\ge 0$,
\begin{equation*}
F_n(t) \;+\; c_0 n^2 \nu \int_0^t e^{- b \nu (t-s)} G_{n-1}(s)\, ds 
\;\le\; r_n(t)\;,
\end{equation*} 
where $r_n(t) = r_n(t,b,C_0)$ is given by
\begin{align*}
r_n(t) \; &=\; C_0 \sum_{k=2}^n F_{k}(0) \Big(\frac{t [n\beta_1(t)]^2}
{\nu}\Big)^{n-k}\, \frac {e^{- b \nu t}}{(n-k)!} \\
& +\; C_0 \Big(\frac{[n\beta_1(t)]^2} {\nu}\Big)^{n-1}\, \int_0^t
e^{- b \nu (t-s)} \, \frac{(t-s)^{n-2}}{(n-2)!}\, F_1(s) \, ds\;.
\end{align*}
\end{asser}

Let $w:\bb R_+ \times [0,1] \to \bb R$ be given by
\begin{equation}
\label{24}
w(t,x) \;=\; [\partial^2_ xd(\rho_t)](x) \;-\; 
\frac 1{\nu} \, \alpha'(t)\;.
\end{equation}

\begin{asser}
\label{s07}
There exist positive constants $0<c_0<C_0<\infty$ such that
\begin{equation*}
\int_0^1 w(t)^2\, dx \;\le\; e^{-c_0\nu t} \int_0^1 w(0)^2\, dx \;+\;
\frac{C_0}{\nu^3} \int_0^t \alpha^{\prime\prime} (s)^2 \, ds 
+\, r_2(t)\;,
\end{equation*}
where the remainder $r_2$ has been introduced in Assertion \ref{s06}.
\end{asser}

\begin{proof}
The proof is similar to the one of Assertion \ref{s04}.
We first show that
\begin{equation*}
\begin{split}
\frac d{dt} \int_0^1 w(t)^2\, dx \; &\le\; -\, \nu \int_0^1 D(\rho_t)
\, [\partial^3_ xd(\rho_t)]^2\, dx \; +\; \frac{1}{A \nu^3} \, \alpha^{\prime\prime} (t)^2\\
& +\; \nu  \int_0^1 \frac 1{D(\rho_t)} \, [\partial_x D(\rho_t)]^2
\, [\partial^2_ xd(\rho_t)]^2\, dx \;
+\;  A \, \nu\, \int_0^1  w(t)^2\, dx 
\end{split}
\end{equation*}
for any $A>0$.  As $w(t)$ vanishes at $x=0$, apply Poincar\'e's to this
functions to get that
\begin{equation*}
\int_0^1  \, w(t)^2\, dx
\;\le\; K_1 \, \int_0^1  \, [\partial^3_ x d(\rho_t)]^2\, dx \;.
\end{equation*}
Hence, choosing $A$ small enough yields
\begin{equation*}
\frac d{dt} \int_0^1 w(t)^2\, dx \;\le\;
-\, c_0 \nu \int_0^1 w(t)^2\, dx 
 \; +\; \frac{C_0}{\nu^3} \, \alpha^{\prime\prime} (t)^2
\;+\; C_0 \nu  G_1(t) \;,
\end{equation*}
where the function $G_1$ has been introduced in \eqref{16}.
We may replace the constant $c_0$ by one which is smaller than the
constant $b_0$ appearing in the statement of Assertion \ref{s06}.
By Gronwall inequality,
\begin{equation*}
\int_0^1 w(t)^2\, dx \;\le\; e^{-c_0\nu t} \int_0^1 w(0)^2\, dx \;+\;
C_0 \int_0^t e^{-c_0\nu(t-s)} \Big\{ \frac{\alpha^{\prime\prime} (s)^2}{\nu^3} \, 
+\, \nu  G_1(s) \Big\}\, ds \;.
\end{equation*}
To complete the proof of the lemma, it remains to recall the statement
of Assertion \ref{s06}.
\end{proof}

\begin{lemma}
\label{s14}
Assume that $c_1\nu > 2K_1C_1\beta_1(t_0)$ for some $t_0>0$. Then,
there exist positive constants $0<c_0<C_0<\infty$ such that for all
$0\le t\le t_0$,
\begin{align*}
\Vert \partial_x d (\rho(t)) \Vert^2_\infty \; & \le\;
C_0 e^{-c_0\nu t} \Big\{ \int_0^1 w(0)^2\, dx  \;+\;
 F_2(0) \;+\; \frac 1\nu e^{C_0 \beta_1(t_0)} \, t\,  F_1(0) \Big\} \\
& \quad +\; \frac{C_0}{\nu^2} \Big\{ \alpha'(t)^2 \;+\;
\frac{1}{\nu} \int_0^t \alpha^{\prime\prime} (s)^2 \, ds
\;+\; \frac{1}{\nu^2} e^{C_0 \beta_1(t_0) [1+t]} \Big\}\;.
\end{align*}
\end{lemma}

\begin{proof}
Assume that
$2K_1C_1\beta_1(t_0) < c_1\nu$ for some $t_0>0$ and fix $0<t\le
t_0$. Since $\int_0^1 (\partial_x d(\rho (t)) \, dx =0$, subtracting
this integral and applying Schwarz inequality, we get that for all
$x\in [0,1]$,
\begin{equation*}
[\partial_x d (\rho(t,x))]^2 \;\le\; C_0 \int_0^1 [\partial^2_x d
(\rho(t,y))]^2 dy\;.
\end{equation*}
Adding and subtracting $\alpha'(t)/\nu$, by Young's inequality, the
previous expression is less than or equal to
\begin{equation*}
C_0 \int_0^1 w(t)^2 dy \;+\; C_0 \frac{\alpha'(t)^2}{\nu^2} \;\cdot
\end{equation*}

By Assertion \ref{s07}, the first term of the previous expression is
less than or equal to
\begin{equation*}
C_0 e^{-c_0\nu t} \int_0^1 w(0)^2\, dx \;+\;
\frac{C_0}{\nu^3} \int_0^t \alpha^{\prime\prime} (s)^2 \, ds \;+\; r_2(t)\;.
\end{equation*}
By its definition and by Assertion \ref{s05}, 
\begin{equation*}
r_2(t) \;\le\; C_0 e^{-c_0\nu t} \Big\{ F_2(0) \;+\; \frac 1\nu 
e^{C_0 \beta_1(t_0)} \, t\,  F_1(0) \Big\} \;+\; \frac{C_0}{\nu^4}
e^{C_0 \beta_1(t_0) [1+t]}\;.
\end{equation*}
To complete the proof it remains to recollect all previous estimates.
\end{proof}

Let
\begin{equation}
\label{21}
q(t)\;=\; \Vert \partial_ x d (\rho_t)\Vert_\infty\;,
\quad Q(t) \;=\; \sup_{0\le s\le t} q(s)\;, 
\quad t\,\ge\, 0\;.
\end{equation}

\begin{asser}
\label{s11}
Suppose that $Q(t_0) (1+n^2K_1) < c_1$ for some $t_0>0$. Then, there
exists a positive constant $C_0<\infty$ such that for all $0\le t\le
t_0$,
\begin{equation*}
\int_0^1 w(t)^{2n}\, dx \;\le\;
e^{-a_n \nu t} \int_0^1 w(0)^{2n}\, dx \;+\; 
\int_0^t e^{-a_n \nu (t-s)}  H_{n-1}(s) \, ds  \;, 
\end{equation*}
where $a_n = [c_1 - Q(t_0) (1+n^2K_1)]/K_1$ and
\begin{equation*}
H_{n-1}(s) \;=\; C_0 \Big\{ n^2 q(s)^2 \, \frac {\alpha'(s)^2}{\nu} \, +
\frac{\alpha^{\prime\prime}(s)^2}{\nu^3} \Big\} \, \int_0^1
w(s)^{2(n-1)} \, dx\;.
\end{equation*}
\end{asser} 

\begin{proof}
Since $w(t)$ vanishes at the boundary, an integration by parts yields
that the time derivative of $\int_0^1 w(t)^{2n}\, dx$  is equal to
\begin{align}
\label{18}
& -2n(2n-1) \, \nu\, \int_0^1 w(t)^{2(n-1)} D(\rho_t) \,
[\partial^3_ xd(\rho_t)]^{2}\, dx \\
& -\; 2n(2n-1) \, \nu\, \int_0^1 w(t)^{2(n-1)} \partial_xD(\rho_t) \,
w(t) \, \partial^3_x d(\rho_t)\, dx \nonumber \\
&-\; 2n(2n-1) \alpha'(t) \, \int_0^1 w(t)^{2(n-1)} \partial_xD(\rho_t) \,
\partial^3_x d(\rho_t)\, dx
\,-\, \frac{2n \alpha^{\prime\prime}(t)}\nu\, 
\int_0^1 w(t)^{2n-1} \, dx \;. \nonumber
\end{align}
In this formula, in the second line, we added and subtracted $(1/\nu)
\alpha'(t)$ to recover $w(t)$ from $\partial^2_x d(\rho_t)$.

Recall the definition of $q(t)$, introduced in \eqref{21}, and the one
of the constants $c_1$, $C_1$, defined in \eqref{26}. Estimating
$\partial_xD(\rho_t)$ by $C_1 q(t)$, and applying Young inequality to
the last three terms of the previous displayed equation, we obtain
that the time derivative of $\int_0^1 w(t)^{2n}\, dx$ is bounded by
\begin{align*}
& -2n\, (2n-1) \, \nu\, \Big\{ c_1 - \frac{ q(t)}2 - \frac 1A \Big\} 
\, \int_0^1 w(t)^{2(n-1)}\,  [\partial^3_ xd(\rho_t)]^{2}\, dx \\
& \quad +\; 2n(2n-1) \, \nu\, \Big\{\frac {q(t)}2 + \frac 1A \Big\} 
\, \int_0^1 w(t)^{2n} \, dx \\
& \qquad +\; C_0 \Big\{ n^2 A q(t)^2 \, \frac {\alpha'(t)^2}{\nu} \, +
\frac{\alpha^{\prime\prime}(t)^2}{\nu^3} \Big\} \, \int_0^1
w(t)^{2(n-1)} \, dx \;,
\end{align*}
for every $A>0$. 

Let $f(x) = w(t,x)^n$. Since $f$ vanishes at the
boundary and since $f'(x) = n w(t,x)^{n-1} \partial^3_ xd(\rho_t)$, by
Poincar\'e's inequality,
\begin{equation*}
\int_0^1 w(t)^{2n} \, dx \;=\; \int_0^1 f^{2} \, dx \;\le\; 
K_1 \int_0^1 [f^{\prime}]^2 \, dx \;=\; 
K_1 n^2 \int_0^1 w(t)^{2(n-1)} \, [\partial^3_ xd(\rho_t)]^{2}\, dx\;.
\end{equation*}
Set $A=2(1+n^2K_1)/c_1$. With this choice, by assumption, $c_1 -
q(t)/2 - 1/A>0$. In particular, the sum of the first two line of the
penultimate displayed equation is less than or equal to
\begin{equation*}
- \frac{2n(2n-1) \, \nu}{2n^2 K_1} \, \Big\{c_1 - q(t) (1+n^2K_1)
\Big\} \, \int_0^1 w(t)^{2n} \, dx \;.
\end{equation*}
Since $2n(2n-1)/2n^2 \ge 1$, to complete the proof it remains to apply
Gronwall inequality.
\end{proof}

\begin{asser}
\label{s08}
Assume that $2(1+K_1) C_1 Q(t_0) <c_1$ for some $t_0>0$. Then,
there exist positive constants $0<c_0<C_0<\infty$ such that for all
$0\le t\le t_0$,
\begin{equation*}
\int_0^1 [\partial^3_ xd(\rho_t)]^2\, dx \;\le e^{-a \nu t} \, 
\int_0^1 [\partial^3_ xd(\rho_0)]^2\, dx \;+\; \int_0^t 
e^{-a \nu (t-s)} H(s)\, ds\;,
\end{equation*}
where $a=[c_1-2(1+K_1) C_1 Q(t_0)]/2K_1$, and
\begin{equation*}
\begin{split}
H(s) \;=\; C_0\, \nu\, q(s)^2 \int_0^1  [\partial^2_x d(\rho_s)]^2\,
dx \; + \; C_0\, \nu\, \int_0^1  [\partial^2_x d(\rho_s)]^4\, dx
\;+\; \frac{C_0}{\nu^3} [\alpha^{\prime\prime}(s)]^2 \;.
\end{split}
\end{equation*}
\end{asser}

\begin{proof}
The proof is similar to the one of Assertion \ref{s05}. Fix $t_0>0$
satisfying the hypothesis of the lemma and consider some $t< t_0$. Since
$\rho(t,1)=\alpha_t$, $\alpha^{\prime\prime}(t) = \nu^2
\, \partial^2_x [D (\rho(t)) \partial^2_x d(\rho(t))]$ at $x=0$,
$1$. Adding and subtracting $\nu^{-1} \alpha^{\prime\prime}(t)$,
and integrating by parts, we have that
\begin{equation}
\label{25}
\begin{split}
\frac 12\, \frac d{dt} \int_0^1  [\partial^3_ xd(\rho_t)]^2 \, dx
\;& =\; -\, \nu \int_0^1  \partial^4_ x d(\rho_t) \,
\partial^2_ x \big\{ D(\rho_t) \,\partial^2_x d(\rho_t)\big\} \, dx \\
& +\; \nu^{-1} \alpha^{\prime\prime}(t) \int_0^1 \partial^4_ x d(\rho_t) \, dx \;.
\end{split}
\end{equation}

Let $D_1 (\alpha) = (\log D)'(\alpha)$, $D_2 (\alpha) = (\log
D)^{\prime\prime}(\alpha)/ D (\alpha)$. Expand $\partial^2_x \{
D(\rho_t) \,\partial^2_x d(\rho_t)\}$, and observe that $\partial^2_ x
D(\rho_t) = D_1(\rho_t) \partial^2_ x d(\rho_t) + D_2(\rho_t)
[\partial_ x d(\rho_t)]^2$ to write the first term on the right hand
side of the previous formula as
\begin{align*}
& -\, \nu \int_0^1  D(\rho_t) \, [\partial^4_ x d(\rho_t)]^2 \, dx
\;-\; 2\nu \int_0^1  \partial_x D(\rho_t) \, \partial^3_ x d(\rho_t)
\,  \partial^4_ x d(\rho_t) \, dx \\
& -\; \nu \int_0^1  D_1(\rho_t)\, 
[\partial^2_x d(\rho_t)]^2 \, \partial^4_ x d(\rho_t) \, dx 
\; -\; \nu \int_0^1  D_2(\rho_t)\, [\partial_x d(\rho_t)]^2 
\partial^2_x d(\rho_t) \, \partial^4_ x d(\rho_t) \, dx \;.
\end{align*}
Recall the definition of $q(t)$ introduced in \eqref{21}, and recall
that $c_1= \inf_{0\le \alpha\le 1} D(\alpha)$, $C_1 = \Vert
D_1\Vert_\infty$. Apply Young's inequality to the last three terms and
to the last term in \eqref{25} to obtain that the left hand side of
\eqref{25} is less than or equal to
\begin{align}
\label{17}
& -\, \nu \,\big[ c_1 - \frac 2A - C_1 q(t) \big] 
\int_0^1  [\partial^4_ x d(\rho_t)]^2 \, dx 
\; + \; C_1 \, \nu\, q(t) \, \int_0^1  [\partial^3_ x d(\rho_t)]^2 \, dx \\
& \qquad + \; A \, C_0\, \nu\, q(t)^2 \int_0^1  [\partial^2_x d(\rho_t)]^2\,
dx \; + \; A \, C_0\, \nu\, \int_0^1  [\partial^2_x d(\rho_t)]^4\, dx
\;+\; \frac{A}{\nu^3} [\alpha^{\prime\prime}(t)]^2
\nonumber
\end{align}
for all $A>0$. 

Since $\partial^2_ x d(\rho(t,1)) = \partial^2_ x d(\rho(t,0))$,
$\int_0^1 \partial^3_ x d(\rho_t) \, dx =0$. Therefore, by
Poincar\'e's inequality,
\begin{equation*}
\int_0^1  [\partial^3_ x d(\rho_t)]^2 \, dx \;\le\;
K_1 \int_0^1  [\partial^4_ x d(\rho_t)]^2 \, dx\;.
\end{equation*}
Set $A= 4/c_1$. Since, by hypothesis, $2 C_1 q(t) \le 2 C_1 Q(t_0) <
c_1$, the first line of \eqref{17} is bounded by
\begin{equation*}
- \, \frac{\nu}{2 K_1} \,\big[ c_1 - 2 C_1 \, [1+K_1] \, q(t) \big]  
\, \int_0^1  [\partial^3_ x d(\rho_t)]^2 \, dx
\;\le\; - a\, \nu \, \int_0^1  [\partial^3_ x d(\rho_t)]^2 \, dx \;,
\end{equation*}
where $a$ has been introduced in the statement of the assertion.  To
complete the proof, it remains to apply Gronwall inequality.
\end{proof}

\begin{proof}[Proof of Proposition \ref{s17}]
The claims are straightforward consequences of Lemmas
\ref{s18} and \ref{s14}. We turn to the third assertion.
\end{proof}

\begin{proof}[Proof of Proposition \ref{s19}]
Since $\partial^2_ xd(\rho) (t,1) - (1/\nu) \alpha'(t)$ vanish as
$x=0$, by Schwarz inequality, for any $x_0\in [0,1]$,
\begin{equation*}
[\partial^2_ xd(\rho_t) (x_0) - \alpha'(t)]^2 \;\le\;
\int_0^1 [\partial^3_ xd(\rho_t) (x)]^2 \;.
\end{equation*}

Fix $t_0>0$.  By Proposition \ref{s17}, $Q(t_0)^2 \le C_0 \delta^2$,
where $\delta^2 = \varepsilon^2 + \nu^{-2}$. Therefore, there exist
$\varepsilon_0>0$ and $\nu_0<\infty$ with the property that the
hypothesis of Assertion \ref{s08} is in force for all
$\varepsilon<\varepsilon_0$, $\nu>\nu_0$. In particular, the previous
expression is bounded by
\begin{equation*}
C_0 \varepsilon^2 \;+\; \int_0^t 
e^{-a \nu (t-s)} H(s)\, ds\;,
\end{equation*}
where $H$ has been introduced in the statement of Assertion
\ref{s08}. By Proposition \ref{s17}, which permits to estimate
$q(s)^2$, by adding and subtracting $\alpha'(s)$ to $\partial^2_
xd(\rho_s)$, which permits to recover the function $w(s)$ introduced
in \eqref{24}, the second term of the previous equation is less than
or equal to
\begin{equation*}
C_0 \Big\{ \frac{1}{\nu^4} + \frac{\varepsilon^2}{\nu^2}\Big\} 
\;+\; C_0 \, \nu \int_0^t ds\, e^{-a \nu (t-s)} 
\int_0^1 \big\{ \delta^2 w(s)^2 + [\partial^2_xd(\rho_s)]^4 \big\}\, dx\;.
\end{equation*}
By Assertion \ref{s07}, this sum is bounded by
\begin{equation*}
C_0 \Big\{ \frac{1}{\nu^4} + \varepsilon^4 \Big\} 
\;+\; C_0 \, \nu \int_0^t ds\, e^{-a \nu (t-s)} 
\Big\{ \delta^2 r_2(s) + \int_0^1 [\partial^2_xd(\rho_s)]^4 \, dx\Big\}\;,
\end{equation*}
By Assertion \ref{s05}, $r_2(s)\le C_0 \delta^2$. We may thus remove
$r_2$ from the previous formula. By Young inequality, by Assertion
\ref{s11} and by Proposition \ref{s17}, the second term without
$r_2(s)$ is less than or equal to
\begin{equation*}
C_0 \Big\{ \frac{1}{\nu^4} + \varepsilon^4 \Big\}  \;+\;
C_0 \delta^2 \int_0^t ds\, e^{-a \nu (t-s)} \int_0^s dr\,
e^{-a' \nu (s-r)} \int_0^1 w(r)^2 \, dx\;.
\end{equation*}
By Assertions \ref{s07} and \ref{s05}, the previous expression is less
than or equal to $C_0 \delta^4$. This concludes the proof of the
proposition. 
\end{proof}

\begin{proof}[Proof of Proposition \ref{s16}]
By Proposition \ref{s17}, for every $t_0>0$, there exists $\nu_0<\infty$ and
$B<\infty$, where $B$ depends on $\alpha(s)$, $0\le s\le t_0$, and on
the initial condition $v_0$, such that for all $0\le t\le t_0$
\begin{equation}
\label{30}
\Vert u_\nu(t)\Vert^2_\infty \;\le\; B \;, \quad
\int_0^1 [\partial_x u_\nu(t)]^2 \, dx \;\le\; B \;
\end{equation}

Fix $t_0>0$ and $0\le t< t_0$. By definition,
$u_\nu(t,0)=u_\nu(t,1)=0$, and 
\begin{equation*}
\partial_t u_\nu \;=\; \nu \, \Big\{ \partial_x \big[ 
D(\alpha_t + \varepsilon u_\nu) \partial_x u_\nu\big] - \alpha'(t)
\Big\}  \;=\; \nu \, \partial_x \Big\{  
D(\alpha_t + \varepsilon u_\nu) \partial_x u_\nu 
- D(\alpha_t) \partial_x v_t \Big\} \;.  
\end{equation*}
where $\varepsilon = \nu^{-1}$.

Therefore, for every $t\ge 0$, an integration by parts yields
\begin{align*}
& \frac 12 \frac d {dt} \int_0^1 [u_\nu (t) - v_t]^2 \, dx \; 
=\; \int_0^1 [u_\nu (t) - v_t] \, \partial_t v_t\, dx \\
&\qquad\qquad -\,
\nu \int_0^1 \partial_x [u_\nu (t) - v_t] \, \big\{
D(\alpha_t + \varepsilon u_\nu) \partial_x u_\nu -
D(\alpha_t) \partial_x v_t\big\}\, dx \;.
\end{align*}
The second term is less than or equal to
\begin{align*}
& - \nu \int_0^1 D(\alpha_t) \, 
[\partial_x u_\nu (t) - \partial_x v_t]^2 \, dx
\;+\; C_0 \int_0^1 |\partial_x u_\nu (t) - \partial_x v_t|
\, |u_\nu| \, |\partial_x u_\nu| \,  dx \\
&\quad \le\;
- c_0 \nu \int_0^1 [\partial_x u_\nu (t) - \partial_x v_t]^2 \, dx
\;+\; \frac {C_0}{\nu}  \int_0^1 u_\nu(t)^2  [\partial_x u_\nu(t)]^2
\, dx \;.
\end{align*}
By \eqref{30}, the second term is bounded by $B/\nu$. Therefore, by
Young's inequality and by Poincar\'e's inequality,
\begin{equation*}
\frac 12 \frac d {dt} \int_0^1 [u_\nu (t) - v_t]^2 \, dx \; 
\le \; -c_0 \nu \int_0^1 [u_\nu (t) - v_t]^2 \, dx 
\; +\; \frac{C_0}\nu \int_0^1 (\partial_t v_t)^2 \, dx
+\, \frac B\nu \;. 
\end{equation*}
To conclude the proof, it remains to apply Gronwall inequality.
\end{proof}

\section{The diffusion coefficient}

We provide in this section of formula for the diffusion coefficient.

Fix a cylinder function $f:\{0,1\}^{\bb Z}\to\bb R$, and recall from
\eqref{40} that $\hat f :[0,1]\to\bb R$ represents the polynomial
defined by
\begin{equation*}
\hat f(\theta) \;=\; E_{\nu_\theta} [ f(\xi)]\;.
\end{equation*}
Write $E_{\nu_{\theta+h}} [ f(\xi)]$ as $E_{\nu_{\theta}} [ f(\xi)
N_h]$, where $N_h$ is the Radon-Nikodym derivative of
$\nu_{\theta+h}$, restricted to the support of $f$, with respect to
$\nu_{\theta}$, to get that
\begin{equation}
\label{61}
\hat f'(\theta) \;=\; \frac 1{\mf c(\theta)}
\sum_{k\in \bb Z} \< f ; \eta(k) \>_{\theta}\;,
\end{equation}
where $\< f ; g \>_{\theta}$ represents the covariance between two
cylinder functions $f$, $g$ in $L^2(\nu_\theta)$: $\< f ; g
\>_{\theta} = E_{\nu_\theta} [ f g] - E_{\nu_\theta} [
f]E_{\nu_\theta} [ g]$, and $\mf c(\theta)$ the static
compressibility, given by $\mf c(\theta)=\theta (1-\theta)$.

Recall the definitions of the cylinder function $h$, introduced in
\eqref{07}. We claim that
\begin{equation}
\label{62}
\hat h' (\theta) \;=\; D(\theta)  \;.
\end{equation}
Indeed, since the cylinder function $c(\eta)$ does not depend
on $\eta(0)$ and $\eta(1)$,
\begin{equation*}
\mf c(\theta)\, D(\theta) \;=\; - \sum_{k\in\bb Z} k \, \big<
[\eta(0)-\eta(1)] c(\eta) \,;\, \eta(k) \big>_\theta \;.
\end{equation*}
Note that all terms in this sum vanish but the one with $k=1$, and
that the sum over $k$ is finite because $c$ is a cylinder function.
By \eqref{06} and by a change of variables, the right hand side is
equal to
\begin{equation*}
- \sum_{k\in\bb Z} k \, \sum_{a=1}^m \sum_{j\in \bb Z} \mu_a(j)\, \big< \tau_{-j} h_a
\,;\, \eta(k) \big>_\theta  \;=\;
- \sum_{a=1}^m \sum_{k,j\in\bb Z} k \,  
\mu_a(j)\, \big< \tau_{-(j+k)} h_a \,;\, \eta(0) \big>_\theta \;.
\end{equation*}
Note that sum over $j$ is finite because $\mu_a$ has finite
support. By definition of $m_a$, and since the total mass of $\mu_a$
vanishes, $\sum_j \mu_a(j)=0$, performing the change of variables
$k'=j+k$ last term becomes
\begin{equation*}
\sum_{a=1}^m m_a \sum_{k,\in\bb Z}  \big< \tau_{-k} h_a \,;\, \eta(0)
\big>_\theta \;=\; \sum_{a=1}^m m_a \sum_{k,\in\bb Z}  \big< h_a \,;\,
\eta(k) \big>_\theta  \;=\; \mf c(\theta) 
\sum_{a=1}^m m_a \hat h'_a (\theta) \;, 
\end{equation*} 
where the last identity follows from \eqref{61}. This last expression
is equal to $\mf c(\theta) \hat h' (\theta)$, which concludes the
proof of \eqref{62}.

\smallskip\noindent{\bf Acknowledgments} The second author would like
to thank L. Bertini, A. De Sole, D. Gabrielli and C. Jona-Lasinio for
introducing him to theory of finite time thermodynamics.

\end{document}